\documentclass[a4paper,11pt,final]{amsart}
\usepackage{amssymb,amsmath}
\usepackage{color}
\usepackage{a4wide}
\usepackage{geometry}
\usepackage{hyperref}
\usepackage{amsmath,amssymb,amsthm}
\usepackage[notref,notcite]{showkeys}

\usepackage{showkeys}

\numberwithin{equation}{section}
\newtheorem{theorem}{Theorem}[section]
\newtheorem{proposition}[theorem]{Proposition}

\newcommand{\N}{{\mathbb N}}

\newcommand{\R}{{\mathbb R}}

\newcommand{\wt}{\widetilde}

\def\bye{\end{document}}
\def\by{\end{proof}\bye}

\def\cP{\mathcal P}

\def\beq{\begin{equation}}
\def\eeq{\end{equation}}

\def\cH{\mathcal H}
\def\wt{\widetilde}

\newtheorem{thm}{\textbf{Theorem}}[section]
\newtheorem{lem}[thm]{\textbf{Lemma}}
\newtheorem{prop}[thm]{\textbf{Proposition}}
\theoremstyle{remark}
\newtheorem{rem}[thm]{\textbf{Remark}}
\newtheorem{cor}[thm]{\textbf{Corollary}}

\newtheorem{exe}[thm]{\textbf{Example}}
\theoremstyle{definition}
\newtheorem{defn}[thm]{{Definition}}
\newtheoremstyle{Claim}{}{}{\itshape}{}{\itshape\bfseries}{:}{ }{#1}
\theoremstyle{Claim}

\newcommand{\C}{\mathcal{C}}

\newcommand{\Ss}{\mathcal{S}}

\newcommand{\USC}{\mathrm{USC}}
\newcommand{\LSC}{\mathrm{LSC}}

\newcommand{\EC}[1]{\overrightarrow{#1}}

\newcommand{\veps}{\varepsilon}

\title[Principal eigenvalue of $k$-Hessians]{Principal eigenvalues for $k$-Hessian operators by maximum principle methods}
\author{Isabeau Birindelli}
\address{Dipartimento di Matematica ``G.\ Castelnuovo''\\ Sapienza Universit\`a di Roma\\ P.le Aldo Moro 2\\ 00185--Roma, Italy}
\email{isabeau@mat.uniroma.it (Isabeau Birindelli)}
\author{Kevin R.\ Payne}
\address{Dipartimento di Matematica ``F. Enriques''\\ Universit\`a di Milano\\ Via C. Saldini 50\\ 20133--Milano, Italy}
\email{kevin.payne@unimi.it (Kevin R. Payne)}\thanks{Payne partially supported by the Gruppo Nazionale per l'Analisi Matematica, la Probabilit\`a e le loro Applicazioni (GNAMPA) of the Istituto Nazionale di Alta Matematica (INdAM) and the projects: GNAMPA 2017 ``Viscosity solution methods for fully nonlinear degenerate elliptic equations'', GNAMPA 2018 ``Costanti critiche e problemi asintotici per equazioni completamente non lineari'' e GNAMPA 2019 ``Problemi differenziali per operatori fully nonlinear fortemente degeneri''.}

\date{\today} 

\keywords{maximum principles, comparison principles, principal eigenvalues, $k$-Hessian operators, $k$-convex functions, admissible viscosity solutions, elliptic sets.}

\subjclass[2010]{35J60, 35J70, 35D40, 35B51, 35P30}
\begin{document}

\maketitle
\vspace{-4ex}
\begin{abstract}
For fully nonlinear $k$-Hessian operators on bounded strictly $(k-1)$-convex domains  $\Omega$ of $\R^N$, a characterization of the principal eigenvalue associated to a $k$-convex and negative principal eigenfunction will be given as the supremum over values of a spectral parameter for which {\em admissible viscosity supersolutions} obey a minimum principle. The admissibility condition is phrased in terms of the natural closed convex cone $\Sigma_k \subset \Ss(N)$ which is an {\em elliptic set} in the sense of Krylov \cite{Kv95} which corresponds to using $k$-convex functions as admissibility constraints in the formulation of viscosity subsolutions and supersolutions. Moreover, the associated principal eigenfunction is constructed by an iterative viscosity solution technique, which exploits a compactness property which results from the establishment of a global H\"older estimate for the unique $k$-convex solutions of the approximating equations.
\end{abstract}

\setcounter{tocdepth}{1}
\tableofcontents

\section{Introduction} For each $1\leq k\leq N$, the $k$-Hessian operator acting on $u \in C^2(\Omega)$ with $\Omega \subseteq \R^N$ open is defined by
\begin{equation}\label{S_kIn}
S_k(D^2u):=\sigma_k(\lambda(D^2u)) := \sum_{1 \leq i_1 < \cdots < i_k \leq N} \lambda_{i_1} (D^2u)\cdots \lambda_{i_k}(D^2u)
\end{equation}
where $\lambda(D^2u)$ indicates the N-vector of ordered eigenvalues of the Hessian matrix $D^2u$ and 
$\sigma_k(\lambda)$ is the elementary symmetric polynomial, which is homogeneous of degree $k$. For example, one has
$$S_1(D^2u)=\Delta u = {\rm tr}(D^2u)\ \quad \mbox{and} \quad S_N(D^2u)= {\rm det}(D^2u).$$
Before describing the scope of this paper, let us mention that, for each $k>1$, 
the $k$-Hessian is fully nonlinear and is (degenerate) elliptic only when constrained in a suitable sense. More precisely, in general one does not have  
\begin{equation}\label{DE_Intro}
	X\leq Y\Rightarrow S_k(X)\leq S_k(Y)
\end{equation}
if $X,Y$ are free to range over all of $\Ss(N)$, the space of $N \times N$ symmetric matrices. However, one will have \eqref{DE_Intro} if one constrains $X$ (and hence $Y$) to belong to 
\begin{equation}\label{gamma_k}
\Sigma_k := \{ A \in \Ss(N): \ \lambda(A) \in \overline{\Gamma}_k\}
\end{equation}
where
\begin{equation}\label{gamma_k}
		\Gamma_k := \{ \lambda \in \R^N: \ \sigma_j(\lambda) > 0 , \ j = 1, \ldots , k \}.
\end{equation}
This leads us to work in the context of admissible viscosity solutions, using the notion of {\em elliptic sets} as introduced by Krylov \cite{Kv95}. This notion has given rise to the development of an organic theory of viscosity solutions with admissibility constraints beginning with Harvey-Lawson \cite{HL09}. In the terminology of \cite{CHLP19}, $\Sigma_k$ is a {\em (constant coefficient) pure second order subequation constraint set} which requires that $\Sigma_k \subset \Ss(N)$ be a proper closed subset satisfying the {\em positivity property}
\begin{equation}\label{P_Intro}
	 A \in \Sigma_k \ \Rightarrow  \ A + P \in \Sigma_k \ \ \text{for each} \ P \in \mathcal{P}, \ \ \cP := \{ P \in \Ss(N): \ P \geq 0 \}
\end{equation} 
and the {\em topological property} 
\begin{equation}\label{T_Intro}
 \Sigma_k = \overline{\Sigma_k^{\circ}},
 \end{equation}
 which is the closure in $\Ss(N)$ of the interior of $\Sigma_k$.

In the nonlinear potential theory language of \cite{HL09}, one would say that $u \in \USC(\Omega)$ is {\em $\Sigma_k$-subharmonic at $x_0$} if $u$ is $k$-convex at $x_0 \in \Omega$ in the sense of Definition \ref{defn:k_convex}. Moreover, one has the following {\em coherence property:} if $u \in \USC(\Omega)$ is twice differentiable in $x_0$ then
	\begin{equation}\label{coherence}  
	\mbox{$u$ is $\Sigma_k$-subharmonic at $x_0 \ \ \Leftrightarrow \ \ D^2u(x_0) \in \Sigma_k$;}
		\end{equation}
that is, the classical and viscosity notions of $k$-convexity coincide at points of twice differentiability. We notice that the reverse implication $(\Leftarrow)$ depends on the positivity property \eqref{P_Intro} of $\Sigma_k$. The topological property \eqref{T_Intro} is sufficient for the local construction of classical strict subsolutions $\varphi$; that is, $S_k(D^2 \varphi) \leq 0$. It also plays and important role in the {\em Harvey-Lawson duality} which leads to an elegant formulation of supersolutions by duality (see Remark 2.12 and Proposition 3.2 of \cite{CHLP19} for 
details).

The main purpose of this paper is to study the {\em principal eigenvalue} associated to the operator $-S_k$ with homogeneous Dirichlet data in bounded domains $\Omega$ by using maximum principle methods. We shall consider $\Omega \subset \R^N$ whose boundaries are of class $C^2$  and {\em $(k-1)$-convex}. The notion of principal eigenvalue that we will define is inspired by that of Berestycki, Nirenberg and Varadhan in their groundbreaking paper \cite{BNV94}.

More precisely, for each $k \in \{1, \ldots, N\}$ fixed, using the notion of $k$-convexity in a viscosity sense, we introduce
$$
\Phi_k^{-}(\Omega) := \{ \psi \in \USC(\Omega): \psi \ \text{is $k$-convex and negative in} \ \Omega\},
$$
We can then consider 
$$
\Lambda_k:= \{ \lambda \in \R: \ \exists \psi \in \Phi_k^{-}(\Omega)\ \text{with} \  S_k(D^2 \psi) + \lambda \psi |\psi|^{k-1} \geq 0 \ \text{in} \ \Omega \},
$$
where the inequality above is again in the admissible viscosity sense, and define our candidate for a (generalized) principal eigenvalue by
$$
\lambda_1^{-}:= \sup \Lambda_k.
$$

We will prove that $\lambda_1^{-}$ is an upper bound for the validity of the minimum principle in $\Omega$; that is, we will show that for any $\lambda<\lambda_1^{-}$ the operator $F[\cdot]=S_k(D^2\cdot)+\lambda   \cdot |\cdot|^{k-1}$ satisfies the minimum principle in $\Omega$ (see Theorem \ref{thm:MPC}). Moreover, in Theorem \ref{thm:psi_1} we will show that $\lambda_1^{-}$ is actually an eigenvalue for the operator $-S_k$ in the sense that there exists $\psi_1<0$ in $\Omega$ such that
$$
\left\{
\begin{array}{lc}
S_k(D^2\psi_1)+\lambda _1^-  \psi_1 |\psi_1|^{k-1}=0 & \mbox{in} \ \Omega\\
\psi_1=0& \mbox{on} \ \partial\Omega.
\end{array}
\right.
$$
Notice that is the linear case $k = 1$ this conforms to the usual notion of eigenvalue. 
If $k$ is odd integer, then $S_k$ is an odd operator and we would have a maximum principle characterization for $\lambda_1^+$; that is, a principal eigenvalue which  corresponds to a positive concave principal eigenfunction.

We should mention that the $k$-Hessians are variational and hence it is possible to give a variational characterization of the principal eigenvalue through a generalized Rayleigh quotient. This was done Lions \cite{PLL85} for $k = N$ and Wang \cite{Wa95} for general $k$.  Using the minimum principle, we prove that both characterizations coincide (see Corollary \ref{cor:PEC}). Hence the existence of the eigenfunction corresponding to the so-called eigenvalue is just a consequence of this equality, using the results in \cite{Wa95, PLL85}. 

Nonetheless, we wish to give a proof of the existence that is independent of the variational characterization. Indeed, the interest in defining the principal eigenvalue by way of \eqref{def:peval1} is twofold. On one hand, it allows one to prove the minimum principle for $\lambda$ below $\lambda_1^{-}$ and on the other hand, it strongly suggests that it may be possible to extend the results to a class of fully non linear operators that are not variational, but which may include the $k$-Hessians. For example,
with $0<\alpha\leq \beta$ and $S_{\alpha,\beta}=\{A\in \Ss(N): \  \alpha I\leq A\leq \beta I\}$, one might consider the operators defined by
$$
	\sigma_k^+(D^2u):=\sup_{A\in S_{\alpha,\beta}}\sigma_k(\lambda(AD^2u)).
$$
This will be the subject of a subsequent paper.

We should also mention that the challenges for a viscosity solution approach to the existence of a solution $\psi_1$ to the eigenvalue problem 
\begin{equation}\label{EVP}
\left\{ \begin{array}{cc}
S_k(D^2 \psi_1) + \lambda_1^{-}\psi_1 |\psi_1|^{k-1} = 0 & {\rm in} \ \Omega \\
\psi_1 = 0 & {\rm on} \ \partial \Omega
\end{array} \right.
\end{equation}
include the lack of global monotonicity in the Hessian $D^2 \psi_1$ and the ``wrong'' 
monotonicity in $\psi_1$. The argument will involve an iterative construction that has its origins in \cite{BD06, BD07} and was used with success in degenerate settings in 
\cite{BGI18}.  At some point in the argument, a compactness property is needed for the sequence of $k$-convex solutions to the approximating equations. The needed compactness property follows from a global H\"older estimate on the sequence of approximating solutions. Unfortunately, using merely maximum principle methods, we are able to prove the estimate only in the case when $k>\frac{N}{2}$. See Remark \ref{rem:GHB} for a discussion of this point, including the use of some measure theoretic techniques (to  augment the maximum principle techniques developed herein) in order establish existence in the nonlinear range $1 < k \leq \frac{N}{2}$. 

\section{Preliminaries.}

In all that follows, $\Omega$ will be a bounded open domain in $\R^N$, $\Ss(N)$ is the space of $N \times N$ symmetric matrices with real entries and 
\begin{equation}\label{evals_A}
\lambda_1(A) \leq \cdots \leq \lambda_N(A)
\end{equation}
are the $N$ eigenvalues of $A \in \Ss(N)$ written in increasing order. The spaces of upper, lower semi-continuous functions on $\Omega$ taking values in $[-\infty, +\infty), (-\infty, +\infty]$ will be denoted by $\USC(\Omega), \LSC(\Omega)$ respectively.

\subsection{$k$-convex functions and $k$-Hessians}

We begin by describing the subset $\Sigma_k \subset \Ss(N)$ which serves to define $k$-convex functions by way of a viscosity constraint on the Hessian of $u \in \USC(\Omega)$. For $k \in \{1, \ldots, N\}$ denote by 
\begin{equation}\label{sigma_k}
\sigma_k(\lambda) := \sum_{1 \leq i_1 < \cdots < i_k \leq N} \lambda_{i_1} \cdots \lambda_{i_k}
\end{equation}
the {\em kth-elementary symmetric function} of $\lambda \in \R^N$ and consider the open set defined by
\begin{equation}\label{gamma_k}
	\Gamma_k := \{ \lambda \in \R^N: \ \sigma_j(\lambda) > 0 , \ j = 1, \ldots , k \},
\end{equation}
which is clearly a cone with vertex at the origin and satisfies
\begin{equation}\label{chain1}
\Gamma^+ \subset \Gamma_N \subset \Gamma_{N-1} \cdots \subset \Gamma_1  
\end{equation}
where $\Gamma^+ := \{ \lambda \in \R^N: \ \lambda_i > 0 \ \ \text{for each} \ i \}$ is the positive cone in $\R^N$ and
\begin{equation}\label{boundary_cones}
	\lambda \in \partial \Gamma_k \ \ \Rightarrow \ \ \sigma_k(\lambda) = 0.
\end{equation}
 
Additional fundamental properties of the cones $\Gamma_k$ are most easily seen by using alternate characterizations of \eqref{gamma_k_garding}. First, the homogeneous polynomials \eqref{sigma_k} are examples of {\em hyperbolic polynomials in the sense of G{\aa}rding}. More precisely, for each fixed $k$, $\sigma_k$ is homogeneous of degree $k$ and is {\em hyperbolic in the direction $a = (1, \ldots, 1) \in \R^N$}; that is, the degree $k$ polynomial in $t \in \R$ defined by  
\begin{equation}\label{hyp_poly}
\sigma_k(ta + \lambda) = \sum_{1 \leq i_1 < \cdots < i_k \leq N} \left(t + \lambda_{i_1} \right) \cdots \left(t + \lambda_{i_k} \right)
\end{equation}
has $k$ real roots for each $\lambda \in \R^N$. The cone $\Gamma_k$ can be defined as 
\begin{equation}\label{gamma_k_garding}
\mbox{ the connected component of $\R^N$ containing $a = (1, \ldots, 1) $ on which $\sigma_k > 0$.}
\end{equation}
This form of $\Gamma_k$ corresponds to G{\aa}rding's original definition for a general hyperbolic polynomial for which the form \eqref{gamma_k} results in the special case of $\sigma_k$. The reader might wish to consult section 1 of Caffarelli-Nirenberg-Spruck \cite{CNS85}.

G{\aa}rding's beautiful theory of hyperbolic polynomials in \cite{Ga59} applied to $\sigma_k$ includes two important consequences; namely {\em convexity}
\begin{equation}\label{gamma_convex}
\mbox{ the cone $\Gamma_k$ is convex}
\end{equation}
and {\em (strict) monotonicity}
\begin{equation}\label{sigma_k_monotonicity}
	\sigma_k(\lambda + \mu) > \sigma_k(\lambda) \ \ \text{for each} \ \lambda, \mu \in \Gamma_k.
\end{equation}
Since $\Gamma^+ \subset \Gamma_k$ for each $k$ and $\Gamma_k$ is a cone, one has that $\Gamma_k$ is a {\em monotonicity cone} for itself; that is,
\begin{equation}\label{monotonicity_cone}
	\Gamma_k + \Gamma_k \subset \Gamma_k
\end{equation}
and, in particular,
\begin{equation}\label{Gamma_positivity}
\Gamma^+ + \Gamma_k \subset \Gamma_k,
\end{equation}
which, with the monotonicity \eqref{sigma_k_monotonicity} for $\mu \in \Gamma^+$, gives rise to the degenerate ellipticity of $k$-Hessian operators as it will be recalled below.

Moreover, Korevaar \cite{Ko87} has characterized $\Gamma_k$ as
	\begin{multline}\label{sigma_k_korevaar}
	\Gamma_k = \{ \lambda \in \R^N: \ \sigma_k(\lambda) > 0,  \frac{\partial \sigma_k}{\partial \lambda_{i_1}}(\lambda) > 0 , \ldots , \frac{\partial^{k-1} \sigma_k}{\partial \lambda_{i_1} \cdots  \partial \lambda_{\i_{k-1}}}(\lambda) > 0  \\
	\ \text{for all}  \ \{1 \leq i_1 < \cdots < \i_{k-1} \leq N\} \}
	\end{multline}
which implies that 
\begin{equation}\label{sum_of_lambdas}
\sum_{j = 1}^{N-k+1} \lambda_{i_j} > 0 \ \ \text{on} \ \Gamma_k \ \ \text{for all} \ \ \{1 \leq i_1 < \cdots < \i_{k-1} \leq N\} \}.
\end{equation}
When $k = N$ this says that each $\lambda_i > 0$ for each $i$ and hence
\begin{equation}\label{Gamma+}
	\Gamma_N = \Gamma^+.
\end{equation}
Additional characterizations of $\Gamma_k$, interesting and useful identities and inequalities involving $\sigma_k$ can be found in \cite{Iv85}, \cite{CNS85}, \cite{Ur90}, \cite{LT94}, \cite{Li96}, \cite{TW99} and \cite{Wa09}. For a modern and self-contained account of G{\aa}rding's theory and its relation to the Dirichlet problem one can consult \cite{HL10} and \cite{HL13}.
 
Clearly the closed G{\aa}rding cone $\overline{\Gamma}_k$ is also convex with 
\begin{equation}\label{chain2}
\overline{\Gamma}^+ = \overline{\Gamma}_N \subset \overline{\Gamma}_{N-1} \cdots \subset \overline{\Gamma}_1  
\end{equation}
and by continuity the monotonicity properties extend to say
\begin{equation}\label{monotonicity_closed}
\overline{\Gamma}_N + \overline{\Gamma}_k \subset \overline{\Gamma}_k \ \ \text{and} \ \ 	\sigma_k(\lambda + \mu) \geq \sigma_k(\lambda) \ \ \text{for each} \ \lambda \in \overline{\Gamma}_k, \mu \in \overline{\Gamma}_N.
\end{equation}
 
For $A \in \Ss(N)$ denote by $\lambda(A):= (\lambda_1(A), \ldots, \lambda_N(A))$ the vector of eigenvalues \eqref{evals_A} and define the {\em $k$-convexity constraint set} by
\begin{equation}\label{gamma_k}
\Sigma_k := \{ A \in \Ss(N): \ \lambda(A) \in \overline{\Gamma}_k\}.
\end{equation}
For $u \in C^2(\Omega)$, one says that $u$ is {\em $k$-convex on $\Omega$} if 
\begin{equation}\label{k_convex_C2}
D^2u(x) \in \Sigma_k \ \ \text{for each} \ x \in \Omega
\end{equation}
and we will say that $u \in C^2(\Omega)$ is {\em strictly $k$-convex on $\Omega$} if
$D^2u(x)$ lies in the interior $\Sigma_k^{\circ}$ of $\Sigma_k$ for each $x \in \Omega$. 
Notice that
\begin{equation}\label{chain3}
	   \Sigma_N \subset \Sigma_{N-1} \cdots \subset \Sigma_1 
\end{equation}
where  
\begin{equation}\label{Theta_1_N}
\Sigma_1 = \mathcal{H} := \{ A \in \Ss(N): \ {\rm tr}(A) \geq 0 \} \ \  \text{and} \ \ \Sigma_N = \mathcal{P} := \{ A \in \Ss(N): \ \lambda_1(A) \geq 0 \};
\end{equation}
that is, $1$-convex functions are classically subharmonic (with respect to the Laplacian) and $N$-convex functions are ordinary convex functions. This consideration carries over to $u \in \USC(\Omega)$ where one defines $k$-convexity by interpreting \eqref{k_convex_C2} in the viscosity sense.

\begin{defn}\label{defn:k_convex}
	Given $u \in \USC(\Omega)$, we say that $u$ is {\em $k$-convex at $x_0 \in \Omega$} if for every $\varphi$ which is $C^2$ near $x_0$
	\begin{equation}\label{k_convex}
	\mbox{ $u - \varphi$ has a local maximum in $x_0 \ \ \Rightarrow \ \ D^2 \varphi(x_0) \in \Sigma_k$.}
	\end{equation}
	We say that $u$ is {\em $k$-convex on $\Omega$} if this pointwise condition holds for each $x_0 \in \Omega$.
	\end{defn} 

\begin{rem}\label{rem:k-convex}
In the viscosity language of superjets, the condition \eqref{k_convex} means
 		\begin{equation}\label{k_convex_jets}
 	(p,A) \in J^{2,+}u(x_0) \ \ \Rightarrow \ \ A \in \Sigma_k,
 	\end{equation}
and there are obviously many equivalent formulations. For example, one can restrict to {\em upper test functions} $\varphi$ which are quadratic and satisfy $\varphi(x_0) = u(x_0)$. See Appendix A of \cite{CHLP19} for a discussion of this point in the context of {\em viscosity subsolutions with admissibility constraints}. Moreover, since $\Sigma_k$ is closed, one also has 
\begin{equation}\label{k_convex_jets_closure}
(p,A) \in \overline{J}^{2,+}u(x_0) \ \ \Rightarrow \ \ A \in \Sigma_k;
\end{equation}
where, as usual, $(p,A) \in \overline{J}^{2,+}u(x_0)$ means that there exists $\{x_k, p_k, A_k\}_{k \in \N}$ such that $(p_k,A_k) \in J^{2,+}u(x_k)$ and $(x_k,p_k,A_k) \to (x_0,p,A)$ as $k \to +\infty$.
	\end{rem}

The $k$-convex constraint $\Sigma_k$ will be used as an {\em admissibility constraint} for the solutions of $k$-Hessian equations considered here. One defines the {\em $k$-Hessian operator} by 
 		\begin{equation}\label{k_Hessian_op}
S_k(A):= \sigma_k( \lambda(A)) \quad \text{for} \ A \in \Ss(N) \ \text{and} \  k \in \{1, \ldots, N\},
\end{equation}
where $\lambda(A)$ is the vector of eigenvalues of $A$ and $\sigma_k$ is given by \eqref{sigma_k}. Notice that 
\begin{equation}\label{S_j_signs}
\mbox{$S_j(A) \geq 0$ \ \ for each \ $A \in \Sigma_k$ \ and each \ $j = 1, \ldots k$}.
\end{equation}
In particular $S_k$ is non-negative on $\Sigma_k$. Important special cases are
 		\begin{equation}\label{1_N_Hessian_op}
S_1(A):= {\rm tr}(A) \quad \text{and} \quad  S_N(A):= {\rm det}(A).
\end{equation}
The following Lemma gives the fundamental structural properties of $\Sigma_k$ and $S_k$. 

\begin{lem}\label{lem:S_Theta_k} For each fixed $k \in \{1, \ldots, N\}$, one has the following properties.
	\begin{itemize}
		\item[(a)]  $\Sigma_k$ is a closed convex cone with vertex at the origin.
		\item[(b)] $\Sigma_k$ is an \underline{elliptic set}; that is, $\Sigma_k \subsetneq \Ss(N)$, is closed, non empty and satisfies the \underline{positivity property}
\begin{equation}\label{P}
			A \in \Sigma_k \ \ \Rightarrow \ \ A + P \in \Sigma_k \ \ \text{for each} \ P \in \mathcal{P}, 
\end{equation}
		where $\mathcal{P}$ are the non negative matrices as defined in \eqref{Theta_1_N}. 
		\item[(c)] $\Sigma_k$ satisfies the \underline{topological property}
\begin{equation}\label{T}
		\Sigma_k = \overline{\Sigma_k^{\circ}}, 
\end{equation}
		where $\Sigma_k^{\circ}:= \{A \in \Ss(N): \lambda(A) \in \Gamma_k \}$ is the interior of $\Sigma_k$.
		\item[(d)] The $k$-Hessian is \underline{increasing along $\Sigma_k$}; that is, 
\begin{equation}\label{DE}
		S_k(A + P) \geq S_k(A) \ \ \text{for each} \ A \in \Sigma_k \ \text{and} \ P \in \mathcal{P}. 
\end{equation}
Moreover, the inequality in \eqref{DE} is strict if $P \in \mathcal{P}^{\circ}$; that is, if $P > 0 \ \text{in} \ \Ss(N)$.
	\end{itemize}
\end{lem}

\begin{proof}
	Part (a) follows from the corresponding properties for $\overline{\Gamma}_k$.
	For the claims in part (b), $\Sigma_k$ is closed by part (a). Each $\Sigma_k$ is non empty since $\Sigma_k \supset \Sigma_N = \mathcal{P}$ as noted in \eqref{chain3} and \eqref{Theta_1_N}. Clearly $\Sigma_1 \subsetneq \Ss(N)$ and hence the same is true for the other values of $k$ by \eqref{chain3}. The property \eqref{T} also clearly holds since $\Sigma_k$ is a closed convex cone with non-empty interior.
	
	For the property \eqref{P}, if $A \in \Sigma_k$ and $P \geq 0$ then $\lambda_i(A+P) \geq \lambda_i(A)$ for each $i= 1, \ldots, N$ and hence
	$$
		\sigma_j(\lambda(A +P)) \geq \sigma_j(\lambda(A)) \geq 0 \ \ \text{for each} \ j = 1, \ldots, k,
	$$  
	and hence \eqref{P}.
\end{proof}

As noted in the introduction, the notion of elliptic sets was introduced by Krylov \cite{Kv95} and starting with the groundbreaking paper Harvey-Lawson \cite{HL09}, has given birth to a nonlinear potential theory approach to viscosity solutions with admissibility constraints. In the terminology of  \cite{CHLP19}, $\Sigma_k$ is a {\em (constant coefficient) pure second order subequation constraint set} which requires that $\Sigma_k \subset \Ss(N)$ be a proper closed subset satisfying the positivity property \eqref{P} and the topological property \eqref{T}. Moreover, being also a convex cone, it is a {\em monotonicity cone subequation} and $\Sigma_k$ is the maximal monotonicity cone for $\Sigma_k$ and its dual $\wt{\Sigma}_k$ as defined below in \eqref{Dirichlet_dual} (see Proposition 4.5 of \cite{CHLP19}). One has reflexivity ($\wt{\wt \Sigma}_k = \Sigma_k$) if topological property \eqref{T} holds (see Proposition 3.2 of \cite{CHLP19}, for example). The following fact, mentioned in the introduction, is worth repeating here.
                     
\begin{rem}\label{rem:coherence}
	In the nonlinear potential theory language of \cite{HL09}, one would say that $u \in \USC(\Omega)$ is {\em $\Sigma_k$-subharmonic at $x_0$} if $u$ is $k$-convex at $x_0 \in \Omega$ in the sense of Definition \ref{defn:k_convex}. Moreover, one has the following {\em coherence property:} if $u \in \USC(\Omega)$ is twice differentiable in $x_0$ then
\begin{equation}\label{coherence}  
	\mbox{$u$ is $\Sigma_k$-subharmonic at $x_0 \ \ \Leftrightarrow \ \ D^2u(x_0) \in \Sigma_k$;}
\end{equation}
	that is, the classical and viscosity notions of $k$-convexity coincide at points of twice differentiability. We notice that the reverse implication $(\Leftarrow)$ depends on the positivity property \eqref{P} of $\Sigma_k$. 
\end{rem}

We now turn to the definition of $\Sigma_k$-admissible viscosity subsolutions and supersolutions for the type of equations involving $k$-Hessian operators $S_k$ that we will treat. The definitions make sense for any constant coefficient pure second order subequation. The main point is to indicate the role of the $k$-convexity constraint $\Sigma_k$ which insures the positivity property for $S_k$, which corresponds to the {\em degenerate ellipticity} of $S_k$ when the Hessian is constrained to $\Sigma_k$. 

\begin{defn}\label{defn:sub_super}
	Let $\Omega \subset \R^N$ be an open set and consider the equation
\begin{equation}\label{Sk_equation}
	S_k(D^2 u) - f(x, u, Du) = 0.
\end{equation} 
	\begin{itemize}
		\item[(a)] A function $u \in \USC(\Omega)$ is said to be a {\em $\Sigma_k$-admissible subsolution of \eqref{Sk_equation} at $x_0 \in \Omega$} if for every $\varphi$ which is $C^2$ near $x_0$
		\begin{align}\label{Sk_subsolution}
		u - \varphi \ \text{has a local maximum in} \ x_0 \ \ \Rightarrow  \ \ S_k(D^2 \varphi(x_0)) - & f(x_0, u(x_0), D\varphi(x_0)) \geq 0  \nonumber \\ & \text{and} \ D^2\varphi(x_0) \in \Sigma_k.
		\end{align}
		\item[(b)] A function $u \in \LSC(\Omega)$ is said to be a {\em $\Sigma_k$-admissible supersolution of \eqref{Sk_equation} at $x_0 \in \Omega$} if for every $\varphi$ which is $C^2$ near $x_0$
		\begin{align}\label{Sk_supersolution}
		u - \varphi \ \text{has a local minimum in} \ x_0 \ \ \Rightarrow  \ \ S_k(D^2 \varphi(x_0)) - & f(x_0, u(x_0), D\varphi(x_0)) \leq 0  \nonumber \\ & \text{or} \ D^2\varphi(x_0) \notin \Sigma_k.
		\end{align}
		\item[(c)] A function $u \in C(\Omega)$ is said to be a {\em $\Sigma_k$-admissible solution of \eqref{Sk_equation} at $x_0$} if both \eqref{Sk_subsolution} and \eqref{Sk_supersolution} hold.
	\end{itemize}
	One says that $u$ is a {\em $\Sigma_k$ admissible (sub-, super-) solution on $\Omega$} if the corresponding statement holds for each $x_0 \in \Omega$. 
\end{defn}

A fundamental example involves $f \equiv 0$.

\begin{exe}\label{exe:fundamental}[$k$-convex and co-$k$-convex functions] By \eqref{Sk_subsolution}, a function $u \in \USC(\Omega)$ is a $\Sigma_k$-admissible subsolution  of
\begin{equation}\label{k_convex_eq}
		S_k(D^2 u) = 0 \ \ \text{in} \ \Omega
\end{equation}
precisely when $u$ is $k$-convex in $\Omega$ (which is equivalent to $u$ being $\Sigma_k$-subharmonic in $\Omega$). On the other hand, $\Sigma_k$-admissible supersolutions of \eqref{k_convex_eq} can be stated in terms of the {\em Dirichet dual} of Harvey-Lawson \cite{HL09}
\begin{equation}\label{Dirichlet_dual}
	\widetilde{\Sigma}_k := - \left( \Sigma^{\circ}_k \right)^c = \left( - \Sigma^{\circ}_k \right)^c,
\end{equation}
where $\widetilde{\Sigma}_k$ is also a constant coefficient pure second order subequation. Using \eqref{Sk_supersolution} and \eqref{Dirichlet_dual}, one can show that $u \in \LSC(\Omega)$ is a $\Sigma_k$-admissible supersolution of \eqref{k_convex_eq} if and only if
\begin{equation}\label{co-k-convex}
	\mbox{ $-u \in \USC(\Omega)$ is $\widetilde{\Sigma}_k$-subharmonic in $\Omega$.}
\end{equation}
One says that $u$ is {\em $\Sigma_k$-superharmonic in $\Omega$} and that $v:= -u$ is a {\em co-$k$-convex} function in $\Omega$.  This claim follows from the Correspondence Principle in Theorem 10.14 of \cite{CHLP19} which in our pure second order situation requires three hypotheses. The first hypothesis is that  $(S_k, \Sigma_k)$ is a {\em compatible operator-subequation pair}; that is, $S_k \in C(\Sigma_k)$ with
\begin{equation}\label{compatible_pair}
	0 = c_0 := \inf_{\Sigma_k} S_k \ \ \text{is finite} \ \ \text{and} \ \ \partial \Sigma_k = \{A \in \Sigma_k: S_k(A) = 0\},
\end{equation}
which follow from the definitions of $S_k$ and $\Sigma_k$. The second hypothesis is that the pair is $\mathcal{M}$-monotone for some convex cone subequation $\mathcal{M}$, which is true for $\mathcal{M} = \mathcal{P}$ in this case. Third hypothesis is that $S_k$ is {\em tolpologically tame} which means that $\{A \in \Sigma_k: S_k(A) = 0\}$ has non-empty interior, which follows from the strict monotonicity of $S_k$ in the interior of $\Sigma_k$.
\end{exe}

A few additional remarks about Definition \ref{defn:sub_super} are in order. First we note that, of course,  there are various equivalent formulations in terms of different spaces of (upper, lower) test functions $\varphi$ for $u$ in $x_0$ in the spirit of Remark \ref{rem:k-convex}. 

\begin{rem}\label{rem:Sk_subsuper} Concerning the $\Sigma_k$-admissibility, notice that:
	\begin{itemize}
		\item [(a)] the part $D^2 \varphi(x_0) \in \Sigma_k$ of the subsolution condition \eqref{Sk_subsolution} is precisely \eqref{k_convex} so that $u$ is automatically $k$-convex in $x_0$;
	\item[(b)] the supersolution condition \eqref{Sk_supersolution} can be rephrased as
	\begin{align}\label{Sk_supersolution2}
	u - \varphi \ \text{has a local minimum in} \ x_0 \ \ \Rightarrow  \ \ S_k(D^2 \varphi(x_0)) - & f(x_0, u(x_0), D\varphi(x_0)) \leq 0  \nonumber \\ & \text{if} \ D^2\varphi(x_0) \in \Sigma_k;
	\end{align}
	that is, it is enough to use lower test functions which are $k$-convex in $x_0$.
\end{itemize}\end{rem}
	
The admissible supersolution definition takes its inspiration from Krylov \cite{Kv95} and was developed in \cite{CP17} for equations of the form $F(x,D^2u) = 0$. In the convex Monge-Amp\`ere case $k=N$ of \eqref{Sk_equation}, an analogous definition was given by Ishii-Lions \cite{IL90}. One good way to understand the supersolution definition \eqref{Sk_supersolution} (or \eqref{Sk_supersolution2}) was pointed out in the convex case in \cite{IL90} and concerns the following coherence property.

\begin{rem}\label{rem:Sk_coherence}
	Suppose that $u \in C^2(\Omega)$ is a classical supersolution in $\Omega$; that is,
\begin{equation}\label{Sk_classical_super}
	S_k(D^2u(x)) + f(x,u(x),Du(x)) \leq 0, \ \ x \in \Omega.
\end{equation}
	If $\varphi \in C^2(\Omega)$ is a lower test function for $u$ in $x_0$ ($u - \varphi$ has local minimum in $x_0$), while one has $D^2 u(x_0) \geq D^2 \varphi(x_0)$ from elementary calculus, one cannot use this to deduce
	$$
	S_k(D^2\varphi(x_0)) + f(x,u(x_0),Du(x_0)) \leq 0
	$$
	\underline{unless} $D^2 \varphi(x_0) \in \Sigma_k$.
\end{rem}	

As a final remark, we note that our main focus will be for the equation
	\begin{equation}\label{Sk_EVequation}
S_k(D^2 u) + \lambda u |u|^{k-1} = f(x).
\end{equation} 
where $k \in \{1, \ldots, N\}$ and $\lambda \in \R$ is a spectral parameter, which will be positive in the interesting cases and associated to \eqref{Sk_EVequation} we will often have a homogeneous Dirichlet condition on $\partial \Omega$.  We will have cause to consider negative and $k$-convex subsolutions to \eqref{Sk_EVequation} as well as non negative supersolutions. Obviously, this means using Definition \ref{defn:sub_super} with $f(x,u,Du) = \lambda u|u|^{k-1}$ where the positivity of $\lambda$ and negativity of $r$ is compatible with the $\Sigma_k$ convexity of (sub)solutions $u$.

\subsection{$(k-1)$-convex domains}

In order to construct suitable barriers for $k$-Hessian operators, we will exploit a suitable notion of strict boundary convexity which is stated in terms of the positivity of the relevant elementary symmetric function of the principal curvatures. More precisely, given $\Omega \subset \R^N$ a bounded domain with $\partial \Omega \in C^2$, we denote by 
\begin{equation}\label{princ_curvs}
	(\kappa_1(y), \ldots , \kappa_{N-1}(y)) \ \ \text{with} \ y \in \partial \Omega
\end{equation}
the {\em principal curvatures} (relative to the inner unit normal $\nu(y)$) which are defined as the eigenvalues of the self-adjoint {\em shape operator} $S$ on the tangent space $T(y)$ defined by 
\begin{equation}\label{shape_operator}
S(X) := - D_X \nu, \ \ X \in T(y).
\end{equation}
 If the boundary is represented locally near a fixed point $y_0 \in \partial \Omega$ as the graph of a suitable function $\phi$, the principle curvatures $\kappa_i(y_0)$ are the eigenvalues of the Hessian of $\phi$ at the relevant point. This will be recalled in the next subsection (as will special coordinate systems well adapted for calculations near the boundary).
 
 The needed concept of convexity is the following notion\footnote{ This is known {\em uniform $(k-1)$-convexity} as in the works of Trudinger beginning with \cite{Tr95}.}  

\begin{defn}\label{defn:k-1_convex_bdy} Let $k \in \{2, \ldots, N\}$.
	$\Omega \subset \R^N$  with $\partial \Omega \in C^2$ is said to be {\em strictly $(k-1)$-convex} if \footnote{ Here and below, we will use the same symbol $\sigma_j$ for the jth-elementary symmetric function on $\R^{N-1}$ and $\R^{N}$.}
	\begin{equation}\label{k-1_convex_bdy}
	\sigma_j(\kappa_1(y), \ldots , \kappa_{N-1}(y)) > 0 \ \ \text{for each} \ y \in \partial \Omega \ \ \text{and each} \ j = 1, \ldots k - 1;
	\end{equation}
	that is, for each $j = 1, \ldots, k - 1$, each {\em jth-mean curvature} is everywhere strictly positive on $\partial \Omega$.  
\end{defn} 
Notice that strict $(N-1)$-convexity is ordinary strict convexity of $\partial \Omega$. One importance of this convexity is that it ensures the existence of functions which are $C^2$, vanish on the boundary and strictly $k$-convex near the boundary. This fact will be used in Proposition \ref{prop:be2} below and depends in part on the following fact.

\begin{lem}\label{lem:k-1_convex_bdy}
	If $\Omega \subset \R^N$ is a bounded strictly $(k-1)$-convex domain with $\partial \Omega \in C^2$, then there exists $R > 0$ such that
	\begin{equation}\label{kCB1}
	\sigma_{j}(\kappa_1(y), \ldots , \kappa_{N-1}(y), R) > 0 \ \ \text{for each} \ y \in \partial \Omega \ \ \text{and each} \ j = 1, \ldots k;
	\end{equation}
	that is, 
	\begin{equation}\label{kCB2}
	(\kappa_1(y), \ldots , \kappa_{N-1}(y), R) \in \Gamma_k \ \ \text{for each} \ y \in \partial \Omega.
	\end{equation}
\end{lem}
\begin{proof}
	With the convention that $\sigma_0(\cdot) = 1$, one has the elementary identity 
\begin{equation}\label{sigma1}
	\sigma_j(\kappa_1, \ldots, \kappa_{N-1}, R) = R \sigma_{j-1}(\kappa_1, \ldots, \kappa_{N-1}) + \sigma_j(\kappa_1, \ldots, \kappa_{N-1}), \ \ j = 1, \ldots, k.
\end{equation}
If $j < k$, then the left hand member is positive on $\partial \Omega$ for any $R \geq 0$ by the convexity assumption \eqref{k-1_convex_bdy}. When $j = k$ and $R > 0$, the first term on the right hand side of \eqref{sigma1} is positive by \eqref{k-1_convex_bdy} and both terms are continuous functions on $\partial \Omega$ which is compact, which gives the claim \eqref{kCB1} for $j = k$ if $R$ is large enough.
\end{proof}

We note that if $\partial \Omega$ is {\em connected} then the conclusion \eqref{kCB1} holds under the weaker convexity assumption 
	\begin{equation}\label{k-1_convex_bdy2}
\sigma_k(\kappa_1, \ldots , \kappa_{N-1}) > 0 \ \ \text{on} \  \partial \Omega.
\end{equation}
See Remark 1.2 of \cite{CNS85} for a proof of this fact, which also makes use of \eqref{sigma1}.

As a final consideration, we make a comparison with the natural notion of strict $\EC{\Sigma}_k$-convexity, as defined in section 5 of Harvey-Lawson \cite{HL09}. This notion is defined in terms of an {\em elliptic cone} $\EC{\Sigma}$ which is an elliptic subset of $\Ss(N)$ (as defined in Lemma \eqref{lem:S_Theta_k} (a)) which is also a pointed cone in the sense that
$$
	A \in \EC{\Sigma} \ \ \Leftrightarrow \ \ t A \in \EC{\Sigma} \ \ \text{for each} \ t \geq 0.
$$
Given an elliptic set $\Sigma$ there is an {\em associated elliptic cone} $\EC{\Sigma}$ which can be defined as the closure of the set 
$$
	\{ A \in \Ss(N): \ \exists t_0 > 0 \ \ \text{with} \ tA \in \Sigma \ \text{for each} \ t \geq t_0 \}
$$
It is easy to see that if $\Sigma$ is an elliptic cone, then $\EC{\Sigma} = \Sigma$.

One says that $\partial \Omega$ is strictly $\EC{\Sigma}$-convex at $x \in \partial \Omega$ if there exists a local defining function $\rho$ for the boundary near $x$ such that \footnote{ More precisely, $\rho \in C^2(B_r(x))$ for some $r > 0$ and $\Omega \cap B_r(x) = \{ y \in B_r(x): \ \rho(y) < 0 \}$ and $D \rho \neq 0$ on $B_r(x)$.}
\begin{equation}\label{HL_convex}
	D^2 \rho(x)_{|T_x \partial \Omega} = B_{|T_x \partial \Omega} \ \ \text{for some} \ B \in \EC{\Sigma}^{\circ},
\end{equation} 
which is to say that $\rho$ is strictly $\EC{\Sigma}$ convex near $x \in \partial \Omega$. In \cite{HL09}, it is shown that solvability of the Dirichlet problem on $\Omega$ for $\Sigma$-harmonic functions holds if $\partial \Omega$ is strictly $\EC{\Sigma}$ and $\EC{\widetilde{\Sigma}}$ convex where $\widetilde{\Sigma} = - \left( \Sigma^{\circ} \right)^c$ is the {\em Dirichlet dual} of $\Sigma$ (as defined in \eqref{Dirichlet_dual}). 

\begin{prop}\label{prop:EC_convexity}
	For $\Omega \subset \R^N$ bounded with $\partial \Omega \in C^2$, one has
	\begin{equation}\label{HL_convex_equiv}
\partial \Omega \  \text{is strictly} \ (k-1)-\text{convex} \ \Longleftrightarrow \ \partial \Omega \ \text{is strictly}\ \EC{\Sigma}_k \ \text{and} \ \EC{\widetilde{\Sigma}} \ \text{convex}.
\end{equation}	
\end{prop}

\begin{proof}
	Since $\Sigma_k$ and $\wt{\Sigma}_k$ are elliptic sets and pointed cones, they are themselves elliptic cones and hence
	\begin{equation}\label{BCE1}
	\Sigma_k = \EC{\Sigma}_k \quad \text{and} \quad \wt{\Sigma}_k = \EC{\wt{\Sigma}}_k.
	\end{equation}
	From \eqref{chains} and \eqref{Theta_1_N} one has
	for each $k \in \{1, \ldots, N \}$
	\begin{equation}\label{BCE2}
	\cP \subset \Sigma_N \subset \Sigma_k \subset \Sigma_1 = \cH
	\end{equation}
	and by the definition of the dual one also has
		\begin{equation}\label{BCE3}
	\cH = \wt{\cH} \subset \wt{\Sigma}_1 \subset \wt{\Sigma}_k \subset \wt{\Sigma}_N = \wt{\cP}
	\end{equation}
	and hence
			\begin{equation}\label{BCE4}
	  \Sigma_k \subset \wt{\Sigma}_k \ \ \text{for each} \ k \in \{1, \ldots, N \}.
	\end{equation}
From \eqref{BCE1} and \eqref{BCE4} one has that $\EC{\Sigma}_k \subset \EC{\widetilde{\Sigma}}_k$ and hence strict $(k-1)$-convexity is precisely what the general Harvey-Lawson theory requires since $\EC{\Sigma}_k^{\circ} = \Sigma_k^{\circ}$.
\end{proof} 

\subsection{Principal coordinate systems and the distance function.}

Consider $\Omega \subset \R^N$ a bounded domain with $C^2$ boundary with principal curvatures $\{ \kappa_i(y)\}_{i=1}^{N-1}$, unit inner normal $\nu(y)$ and tangent space $T(y)$ at each $y \in \partial \Omega$. Denote the distance function to the boundary by
\begin{equation}\label{distance}
	d(x):= {\rm dist}(x, \partial \Omega), \ \ x \in \R^N.
\end{equation}
Following section 14.6 of \cite{GT}, will recall some known facts concerning the calculation of $\kappa_i(y_0)$ at a fixed boundary point $y_0$ and the notion of a principal coordinate system near $y_0$ which yields nice formula for the Hessian of $d$ in suitable tubular neighborhoods of the boundary.

With $y_0 \in \partial \Omega$ fixed, choose coordinates $x = (x',x_N) \in \R^{N-1} \times \R = \R^N$ such that the inner unit normal is $\nu(y_0) = (0,1)$. Then there exists an open neighborhood $\mathcal{N}_0$ of $y_0$ and a function 
 \begin{equation}\label{defining_function}
 \mbox{$\phi: \mathcal{N}_0 \cap T(y_0) \to \R$ of class $C^2$ with $ D \phi(y_0^{\prime}) = 0$}
 \end{equation}
so that  
\begin{equation}\label{local_graph}
	\partial \Omega \cap \mathcal{N}(y_0) = \{(x',\phi(x'): \ \ x' \in \mathcal{N}_0 \cap T(y_0) \}
\end{equation} 
and 
\begin{equation}\label{princ_curvs2}
\mbox{the principal curvatures $\{ \kappa_i(y_0)\}_{i=1}^{N-1}$ \ \ are the eigenvalues of \ $D^2 \phi(y_0^{\prime})$}.
\end{equation} 
In a {\em principal coordinate system at $y_0$}, where one takes the axes $x_1, \ldots, x_{n-1}$ along the associated eigenvectors for $D^2 \phi(y_0^{\prime})$, one has
\begin{equation}\label{princ_curvs3}
D^2 \phi(y_0^{\prime}) = {\rm diag} \left[ \kappa_1(y_0), \ldots , \kappa_{N-1}(y_0) \right].
	\end{equation}  
	
The following properties of bounded $C^2$ domains are well known and will be used repeatedly in the sequel. For the proofs, see Lemma 14.16 and Lemma 14.17 of \cite{GT}. 

\begin{lem}\label{lem:good_tubes}
	Let $\Omega \subset \R^N$ be a bounded domain with $C^2$ boundary. Then there exists $\delta > 0$ such that:
	\begin{itemize}
		\item [(a)] $\partial \Omega$ satisfies a uniform interior (and uniform exterior) sphere condition with balls of radius bounded below by $\delta$ so that the principal curvatures satisfy for each $i \in \{1, \ldots, N-1 \}$
		\begin{equation}\label{GT_curv}
		|\kappa_i(y)| \leq \frac{1}{\delta} \ \ \text{for each} \ y \in \partial \Omega;
		\end{equation}
		\item[(b)] the distance function $d(\cdot):= {\rm dist}(\cdot, \partial \Omega)$ satisfies
			\begin{equation}\label{GT_C2}
		d \in C^2(\overline{\Omega}_{\delta})
		\end{equation}
		and 
		\begin{equation}\label{GT_nablad}
		\mbox{ $|D d(x)|= 1$ for each $x \in \Omega_{\delta}$, }
		\end{equation}
		where
			\begin{equation}\label{Omega_delta}
		\Omega_{\delta} := \{x \in \Omega,\ 0 < d(x) < \delta \};
		\end{equation}
		\item[(c)] for each $x \in \Omega_{\delta}$
		\begin{equation}\label{GT_min_dist}
		\mbox{there exists a unique $y = y(x) \in \partial \Omega$ such that $d(x) = |x - y|$;}
		\end{equation}
		\item[(d)] from \eqref{GT_nablad} one has that $D^2d(x_0)$ has a zero eigenvalue associated to the eigenvector $D d(x_0)$ for each $x_0 \in \Omega_{\delta}$ and using a principal coordinate system based at the point $y_0 = y(x_0)$, which realizes the distance from $x_0$ to the boundary, one has $D d(x_0) = (0, \ldots, 0, 1)$ and
		\begin{equation}\label{GT_hess_d}
		D^2d(x_0) = {\rm diag} \left[ \frac{-\kappa_1}{1 - \kappa_1 d}, \ldots, \frac{-\kappa_{N-1}}{1 - \kappa_{N-1}d}, 0 \right],
		\end{equation} 
		where $\kappa_i = \kappa_i(y_0)$, $d = d(x_0)$ and $1 - \kappa_i d > 0$ since $d < \delta$ and $\kappa_i$ satisfies \eqref{GT_curv}.
	\end{itemize} 
	\end{lem}

\subsection{Computing $S_k$ in special coordinate systems}

Managing $S_k$ is facilitated by using the principal coordinate systems near the boundary discussed above. Also radial functions are often handy for comparison arguments used in Hopf-type boundary estimates and H\"older regularity arguments, as we will see. In this subsection, we record two lemmas for future use.

\begin{lem}\label{lem:PCS}
	Let $\Omega \subset \R^N$ be a bounded domain with $C^2$ boundary. For any $g \in C^2(\Omega)$ and any $x_0 \in \Omega_{\delta}$ one has the following formula for the composition $v = g \circ d$ and for each $j = 1, \ldots, N$
	\begin{equation}\label{S_k_composition}
	S_j(D^2v(x_0)) = \sigma_j \left( \frac{-\kappa_1 }{1 - \kappa_1 d} g^{\prime}(d), \ldots, \frac{-\kappa_{N-1}}{1 - \kappa_{N-1}d} g^{\prime}(d), g^{\prime\prime}(d) \right),
	\end{equation}
	where again $\kappa_i = \kappa_i(y_0)$ and $d = d(x_0)$ in a principal coordinate system based at $y_0 \in \partial \Omega$ which realizes the distance to $x_0 \in \Omega_{\delta}$ as in Lemma \ref{lem:good_tubes}.
	\end{lem}

\begin{proof}
For $g \in C^2$ the Hessian of the composition $v = g \circ d$ in $\Omega_{\delta}$ is given by
\begin{equation}\label{dist}
D^2v = g^\prime(d)D^2d + g^{\prime\prime}(d) D d\otimes D d
\end{equation}
which has eigenvalues $\lambda_N(D^2v)=g^{\prime\prime}(d)$ and $\lambda_i(D^2v)=g^\prime(d)e_i(d)$ where $\{e_i(d)\}_{i=1}^{N-1}$ are the first $N-1$ eigenvalues of $D^2 d$ whose expression at $x_0 \in \Omega_{\delta}$ in a principal coordinate system based at $y_0 = y(x_0)$ is given by \eqref{GT_hess_d} and hence
		\begin{equation}\label{GT_hess_god}
D^2v(x_0) = {\rm diag} \left[ \frac{-\kappa_1}{1 - \kappa_1 d} g'(d), \ldots, \frac{-\kappa_{N-1}}{1 - \kappa_{N-1}d} g'(d), g^{\prime \prime}(d) \right],
\end{equation} 
from which \eqref{S_k_composition} follows by the definition of $S_j$.
\end{proof}

\begin{lem}\label{lem:radial}
For radial functions $w(x)=h(|x - x_0|)$ with $h\in C^2$, the eigenvalues of $D^2w$ in any punctured neighborhood of $x_0$ are
\begin{equation}\label{evals_radial}
\mbox{$h^{\prime\prime}(r)$ with multiplicity one and $h^{\prime}(r)/r$ with multiplicity $N-1$,}
\end{equation}
where $r:= |x - x_0|$ and hence
\begin{equation}\label{S_k_radial1}
S_k(D^2w(x))=h^{\prime\prime}(r)\left(\frac{h^\prime(r)}{r}\right)^{k-1}\left(\begin{array}{c} N-1\\k-1\end{array}\right)+
\left(\frac{h^\prime(r)}{r}\right)^{k}\left(\begin{array}{c} N-1\\k\end{array}\right);
\end{equation}
that is,
\begin{equation}\label{S_k_radial2}
S_k(D^2w(x))=\left(\frac{h^\prime(r)}{r}\right)^{k-1}\left(\begin{array}{c} N-1\\k-1\end{array}\right)\left[h^{\prime\prime}(r)+\frac{h^\prime(r)}{r}\frac{N-k}{k}\right].
\end{equation}
\end{lem}

\begin{proof}
	The claim in \eqref{evals_radial} is well known, from which \eqref{S_k_radial1} and \eqref{S_k_radial2} follow easily.
\end{proof}

\section{Comparison and maximum principles.}

As suggested in the title, we will make use of various comparison and maximum principles for $\Sigma_k$-admissible viscosity subsolutions and supersolution in the sense of Definition \ref{defn:sub_super} and the subsequent remarks and examples. While they will be special cases of the results in \cite{HL09}, \cite{CP17} and \cite{CHLP19}, for the convenience of the reader we will give the precise statements and some indication of the proofs. In all that follows $\Omega$ will be an open bounded domain in $\R^N$.

We begin the most basic comparison result, which concerns a  $\Sigma_k$-subharmonic and $\Sigma_k$-superharmonic pair, as presented in Example \ref{exe:fundamental}.

\begin{thm}\label{thm:cp1} Suppose that $u \in \USC(\overline{\Omega})$ and $v \in \LSC(\overline{\Omega})$ are a $\Sigma_k$-admissible viscosity subsolution/supersolution pair for the homogeneous equation $S_k(D^2 u) = 0$ in $\Omega$. Then the comparison principle holds; that is,
	\begin{equation}\label{cp1}
		u \leq v \ \ \text{on} \ \partial \Omega \ \ \Rightarrow \ \ u \leq v \ \ \text{on} \ \Omega .
		\end{equation}
	\end{thm}
\begin{proof} The hypothesis is equivalent to saying that $u$ and $v$ are $\Sigma_k$-subharmonic and $\Sigma_k$-superharmonic in $\Omega$, as discussed in Example \ref{exe:fundamental}. Since $\Sigma_k$ is a pure second order subequation, one has the  comparison principle \eqref{cp1} as a corollary of the comparison principle of \cite{HL09} (see also Theorem 9.3 of \cite{CHLP19}). The main ingredients in the proof are that $-v$ is $\wt{\Sigma}_k$-subharmonic and that $w:= u - v$ is $\wt{\mathcal{P}}$-subharmonic (coming from the $\mathcal{P}$-monotonicity of $\Sigma_k$ and its dual), for which the {\em zero maximum principle} holds
	\begin{equation}\label{zmp}
w \leq 0 \ \ \text{on} \ \partial \Omega \ \ \Rightarrow \ \ w \leq 0 \ \ \text{on} \ \Omega .
\end{equation}	
See section 7 of \cite{CHLP19} for details in the general case of $\wt{\mathcal{M}}$-monotone subequations.
	\end{proof}

Notice that in Theorem \ref{thm:cp1}, $u$ is $k$-convex and $-v$ is co-$k$-convex, as as described in Example \ref{exe:fundamental}. A simple corollary is the zero maximum principle for $k$-convex functions.

\begin{cor}\label{cor:cp1} If $u \in \USC(\overline{\Omega})$ is $k$-convex in $\Omega$ then 
	\begin{equation}\label{zmp2}
	u \leq 0 \ \ \text{on} \ \partial \Omega \ \ \Rightarrow \ \ u \leq 0 \ \ \text{on} \ \Omega .
	\end{equation}
\end{cor}

\begin{proof}
	It suffices to apply Theorem \ref{thm:cp1} with the smooth function $v \equiv 0$. Since $\wt{\Sigma_k}$ also satisfies the positivity property \eqref{P}, the coherence property of Remark \ref{rem:coherence} holds and hence $v \equiv 0$ is $\Sigma_k$-superharmonic since $0 = S_k(-v) \in \wt{\Sigma_k}$. 
\end{proof}

More is true. One has the strong maximum principle for $\Sigma_k$-subharmonic functions. 

\begin{thm}\label{thm:SMP}
For each $u \in \USC(\Omega)$ which is $\Sigma_k$-subharmonic ($k$-convex) on a bounded domain (open, connected set) one has
	\begin{equation}\label{SMP}
	\mbox{if there exists $x_0 \in \Omega$ with $\displaystyle{u(x_0) = \sup_{\Omega}u}$, \ then $u$ is constant in $\Omega$.}
	\end{equation}
	In particular, for $u$ as in Corollary \ref{cor:cp1}, one has either $u < 0$ in $\Omega$ or $u \equiv 0$ on $\overline{\Omega}$.
\end{thm}

\begin{proof}
	As noted in \eqref{Theta_1_N}, the admissibility constraint sets satisfy 
	$$ 
	\Sigma_k \subset \Sigma_1 = \mathcal{H} \ \  \text{for each} \ k = 1, \ldots, N
	$$
	and hence each $u$ which is $\Sigma_k$-subharmonic on $\Omega$ will be $\mathcal{H}$-subharmonic on $\Omega$ and hence satisfies the mean value inequality
	$$
	u(x_0) \leq \frac{1}{|B_r(x_0)|} \int_{B_r(x_0)} u(x) \, dx
	$$
for each $B_r(x_0) \subset \Omega$, from which \eqref{SMP} follows. 
	\end{proof}

Our next comparison result is tailored for some of the pointwise estimates we will need.

\begin{thm}\label{thm:cp2} Let $c \geq 0$ be fixed. Suppose that $u \in \LSC(\overline{\Omega})$ satisfies \footnote{ By this we mean that $u$ is a $\Sigma_k$-admissible viscosity supersolution of the equation $S_k(D^2u) - c = 0$.} 
	\begin{equation}\label{S_k_c}
		S_k(D^2 u) \leq c \ \  \text{in} \ \Omega. 
		\end{equation}
		Suppose that $v \in C^2(\Omega) \cap C(\overline{\Omega})$ is a strictly $k$-convex strict subsolution, that is, 
	\begin{equation}\label{v_strict}
		D^2 v(x) \in \Sigma_k^{\circ} \ \ \text{and} \ \ S_k(D^2v(x)) > c \ \ \text{for all} \ x \in \Omega.
	\end{equation}
Then, one has the comparison principle
	\begin{equation}\label{cp1}
	v \leq u \ \ \text{on} \ \partial \Omega \ \ \Rightarrow \ \ v \leq u \ \ \text{on} \ \Omega .
	\end{equation}
\end{thm}

\begin{proof}
Suppose not, then $v - u \in \USC(\overline{\Omega})$ will have a (positive) maximum at some interior point $x_0 \in \Omega$. Hence $u - v$ will have a (negative) minimum at $x_0$. Choose $\varphi = v$ in Definition \ref{defn:sub_super} (b) of a $\Sigma_k$-admissible supersolution $u$ to find
$$
S_k(D^2v(x_0)) \leq c \ \ \text{or} \ \  D^2 v(x) \not\in \Sigma_k^{\circ},
$$
which contradicts \eqref{v_strict}.
\end{proof}

Some variants of these principles will also be present in some of the proofs.

\section{Boundary estimates}

For the minimum principle characterization of Theorem \ref{thm:MPC} and for the global H\"{o}lder regularity result of Theorem \ref{thm:GHB}, we will make use of various barrier functions which provide some needed one-sided bounds near the boundary $\partial \Omega$ of bounded $C^2$ domains. The arguments are standard, but the details involve having a sufficiently robust calculus for the $k$-Hessian. 

The first estimate is a form of the Hopf lemma which will be applied to the subsolutions $\psi$ competing in the Definition of the principal eigenvalue when we prove the minimum principle characterization of Theorem \ref{thm:MPC}. 
%\newpage

\begin{proposition}\label{prop:be1}
Given $\lambda \geq 0$. Suppose that $\psi \in \USC(\overline{\Omega})$ is a $k$-convex subsolution i.e.  a $\Sigma_k$-admissible subsolution  in the sense of Definition \ref{defn:sub_super} of
\begin{equation}\label{EV_problem}
\left\{ \begin{array}{lc}
S_k(D^2\psi) + \lambda\psi|\psi|^{k-1} = 0 &\mbox{in}\ \Omega\\
\psi=0 &\mbox{on}\ \partial\Omega
\end{array}
\right. 
\end{equation}
which is negative on $\Omega$. Then 
\begin{equation}\label{barrier_estimate1}
	\mbox{ there exists \ $C_1>0$ \ such that \ $\psi(x)\leq -C_1 \, d(x)$ \ for all $x \in \Omega_{\delta/2}$,}
	\end{equation}
where $\Omega_{\delta/2}$ is the tubular neighborhood defined as in  \eqref{Omega_delta} with $\delta > 0$ such that the conclusions of Lemma \ref{lem:good_tubes} hold in the larger neighborhood $\Omega_{\delta}$.
\end{proposition}
\begin{proof}
One uses the usual argument of comparing $\psi$ with a smooth radial function $w$ which is a strict supersolution of \eqref{EV_problem} and dominates $\psi$ on the boundary of an annular region within the good tubular neighborhood $\Omega_{\delta}$ (see Proposition 2.5 of \cite{BGI18}, for example).
For completeness, we give the argument.

For each $x_0 \in \Omega_{\delta/2}$, denote by $y_0 = y_0(x_0) \in \partial \Omega$ which realizes the distance to the boundary and by $z_0 = z_0(x_0) = \delta(x_0 - y_0)/|x_0 - y_0|$ the center of an interior ball such that 
\begin{equation}\label{inderior_sphere}
	B_{\delta}(z_0) \subset \Omega \quad  \text{and} \quad \overline{B}_{\delta}(z_0) \cap  (\R^N \setminus \Omega ) = \{y_0\}.
\end{equation}
Consider the smooth negative radially symmetric function 
\begin{equation}\label{radial_barrier}
    w(x)=C_0(e^{-m \delta} - e^{-m|x-z_0|}) \ \ \mbox{for each} \ x \in \mathcal{A}:= B_{\delta}(z_0) \setminus \overline{B}_{\delta/2}(z_0),
\end{equation}
where $C_0$ and $m$ are chosen to satisfy
\begin{equation}\label{pick_C0}
C_0 \leq \frac{\sup_{\Omega \setminus \Omega_{\delta/2}} \psi}{e^{-m\delta}-e^{-m \delta/2}} 
\end{equation}
and 
\begin{equation}\label{pick_m}
    m > \frac{2(N  - k)}{k \delta}.
\end{equation}
One has $w(x)\geq \psi(x)$ on $\partial \mathcal{A}$; which is trivial on the outer boundary where $w$ vanishes and  the choice of $C_0$ in \eqref{pick_C0} is used on the inner boundary which is contained in $\Omega \setminus \Omega_{\delta/2}$. The key point is to show that
\begin{equation}\label{barrier_comparison}
	w(x) \geq \psi(x) \ \text{on} \ \mathcal{A}.
	\end{equation}
%which one shows by contradiction. 
Suppose, by contradiction, that $\psi-w$ has a positive maximum point at $\bar{x}$ in the interior of $\mathcal{A}$. By Definition \ref{defn:sub_super}(a) of the $\Sigma_k$-admissible subsolution $\psi$ of the PDE, one has
$$
    S_k(D^2w(\bar{x}))\geq -\lambda\psi(\bar{x})|\psi(\bar{x})|^{k-1}\geq 0
$$
But, using the radial calculation \eqref{S_k_radial2} with $h(r) = C_0(e^{- m \delta} - e^{-mr})$ and $r = x - z_0$, with the choice on $m$ in \eqref{pick_m} where $r > \delta/2$, one has
$$
S_k(D^2w(x))=\frac{C_0^k m^ke^{-mkr}}{r^k}\left(\begin{array}{c} N-1\\k-1\end{array}\right)\left[-mr+ \frac{N-k}{k} \right]<0
$$
which yields a contradiction. Hence, for $x_0 \in \Omega_{\delta/2}$ one has
$$
    \psi(x_0) \leq w(x_0) = C_0 e^{-m \delta} \left(1 - e^{m d(x_0)} \right) \leq  C_0 e^{-m \delta} (-m d(x_0))
$$
and one can take $C_1 =- mC_0e^{-m \delta}$. 
\end{proof}

The next estimate gives a lower bound near the boundary for the supersolutions treated in the minimum principle characterization of Theorem \ref{thm:MPC}.

\begin{proposition}\label{prop:be2}
Let $\Omega$ be a bounded strictly $(k-1)$-convex domain and $\lambda \geq 0$. Suppose that $u \in \LSC(\overline{\Omega})$ is a $\Sigma_k$-admissible supersolution \footnote{ By this we mean that $u$ is a $\Sigma_k$-admissible supersolution of $S_k(D^2 u) + \lambda u |u|^{k-1} = 0$ in $\Omega$ in the sense of Definition \ref{defn:sub_super}(b) and that $u \geq 0$ on $\partial \Omega$.} of 
\begin{equation}\label{EV_problem2}
\left\{\begin{array}{lc}
S_k(D^2u) + \lambda u|u|^{k-1} = 0 &\mbox{in}\ \Omega\\
u = 0 &\mbox{on}\ \partial\Omega
\end{array}
\right. .
\end{equation}
Then 
\begin{equation}\label{barrier_estimate2}
\mbox{there exist \ $C_2>0$ \ and \ $d_0 >0$ \ such that \  $u(x)\geq -C_2 \, d(x)$ \ for all $x \in \Omega_{d_0}$,}
\end{equation}
where, as always, $\Omega_{d_0}:= \{ x \in \Omega: \ 0 < d(x) < d_0 \}$. 
\end{proposition}
Before giving the proof, a pair of remarks are in order. Since $u(x) \geq 0 = -C_2d(x)$ for each $x \in \partial \Omega$, the lower bound in \eqref{barrier_estimate2} holds trivially there.  Moreover, having $u > 0$ at boundary points facilitates having a negative lower bound and the interesting case concerns $u \equiv 0$ on $\partial \Omega$.

\begin{proof}[Proof of Proposition \ref{prop:be2}:] Since $\partial \Omega \in C^2$, by Lemma \ref{lem:good_tubes}, there exists $\delta > 0$ such that $d \in \C^2(\overline{\Omega}_{\delta})$. Consider the comparison function $\varphi \in C^2(\overline{\Omega}_{\delta})$ defined by
\begin{equation}\label{def_test_phi}
	\mbox{$\varphi(x):= e^{-t d(x)} - 1 := g(d(x))$ \ \ with $t > 0$ sufficiently large.}
\end{equation}

\noindent  
{\bf Claim:} {\em For $t$ sufficiently large and $d_0$ sufficiently small, one has:}
	\begin{equation}\label{phi_k_convex}
	D^2 \varphi(x_0) \in \Sigma_k^{\circ} \ \ \text{\em for each} \ x_0 \in \Omega_{d_0};
	\end{equation}
	\begin{equation}\label{phi_bdy}
		\varphi \leq u \ \ \text{\em on} \ \partial \Omega_{d_0};
		\end{equation}
		\begin{equation}\label{phi_sub}
		S_k(D^2 \varphi) + \lambda \varphi |\varphi|^k> 0 \ \ \text{\em in} \ \Omega_{d_0}
		\end{equation}
	{\em That is, on a sufficiently small tubular domain $\Omega_{d_0}$, the function $\varphi$ is a $C^2$ strictly $k$-convex strict subsolution of the eigenvalue equation which is dominated by the supersolution $u$ on $\partial \Omega_{d_0}$.}

Given the claim, the $\Sigma_k$-admissible supersolution $u$ must satisfy
\begin{equation}\label{u_dominates_phi}
	\mbox{$u \geq \varphi$ on $\Omega_{d_0}$.} 
\end{equation}
Indeed, if \eqref{u_dominates_phi} were to be false, then $u - \varphi \in \LSC(\overline{\Omega}_{d_0})$ would have its (negative) minimum at some $ x_0 \in \Omega_{d_0}$ (using \eqref{phi_bdy}). By  Definition \ref{defn:sub_super}(b), one would then have
$$
S_k(D^2 \varphi(x_0)) + \lambda \varphi(x_0) |\varphi(x_0)|^k \leq 0 \ \ \text{or} \ \ D^2 \varphi(x_0) \not\in \Sigma_k^{\circ},
$$ 
which contradicts \eqref{phi_sub}-\eqref{phi_k_convex}. The relation \eqref{u_dominates_phi} gives the barrier estimate \eqref{barrier_estimate2} since 
$$
	u(x) \geq \varphi(x):= e^{-t d(x)} - 1   \geq -C_2 d(x), \ \ x \in \Omega_{d_0}
$$
if $C_2$ is chosen to satisfy $C_2 \geq t$. 

Thus it remains only to verify the Claim. We begin with the strict $k$-convexity of $\varphi$ claimed in \eqref{phi_k_convex}. Using Lemma \ref{lem:good_tubes} on $\varphi = g \circ d$ and calculating the needed derivatives of  $g(d) := e^{-td} - 1$, on $\Omega_{\delta}$ one has for each $j = 1, \ldots, k$
\begin{equation}\label{S_j}
	S_j(D^2 \varphi(x_0)) = t^j e^{-jtd(x_0)} \sigma_j \left( \frac{\kappa_1(y_0) }{1 - \kappa_1(y_0) d(x_0)} , \ldots, \frac{\kappa_{N-1}(y_0)}{1 - \kappa_{N-1}(y_0)d(x_0)} , t \right),
\end{equation}
where again we use a principal coordinate system based at $y_0 = y(x_0) \in \partial \Omega$ which realizes the distance to $x_0 \in \Omega_{\delta}$. Now, using the strict $(k-1)$-convexity of $\partial \Omega$, by Lemma \ref{lem:k-1_convex_bdy} there exists $R > 0$ and $\gamma_0 > 0$ so that
\begin{equation}\label{kCB1_recall}
\sigma_{j}(\kappa_1(y_0), \ldots , \kappa_{N-1}(y_0), R) \geq \beta_0 > 0 \ \ \text{for each} \ y_0 \in \partial \Omega \ \ \text{and each} \ j = 1, \ldots k.
\end{equation}
We have used the continuity of each $\sigma_j$ and the compactness of $\partial \Omega$ to pick up the positive lower bound $\beta_0$. Since each $\sigma_j$ (with $1 \leq j \leq k$) is increasing in $\Gamma_k$ with respect to $\lambda_N$, we can freely replace $R$ by any $t \geq R$ in \eqref{kCB1_recall}. Again by continuity and compactness, we can choose $d_0 \leq \delta$ so that for each $x_0 \in  \Omega_{d_0}$ and for each $j = 1, \ldots k$, one has
\begin{equation}\label{kCB2}
\sigma_{j}\left( \frac{\kappa_1(y_0) }{1 - \kappa_1(y_0) d(x_0)} , \ldots, \frac{\kappa_{N-1}(y_0)}{1 - \kappa_{N-1}(y_0)d(x_0)} , R \right) \geq \frac{\beta_0}{2} > 0.
\end{equation}
Indeed, with
\begin{equation}\label{kCB3}
\mu := \max_{y_0 \in \partial \Omega} \max_{1 \leq i \leq N - 1} |\kappa_i(y_0)|
\end{equation}
\begin{equation*}\label{kCB4}
p := \left( \frac{\kappa_1(y_0) }{1 - \kappa_1(y_0) d(x_0)} , \ldots, \frac{\kappa_{N-1}(y_0)}{1 - \kappa_{N-1}(y_0)d(x_0)} \right) \ \ \text{and} \ \  q:= \left(\kappa_1(y_0), \ldots , \kappa_{N-1}(y_0) \right) ,
\end{equation*}
by choosing $d_0 \leq 1/(2 \mu)$, for every $x_0 \in \overline{\Omega}_{d_0}$ one has
\begin{equation*}\label{kCB5}
\left| p - q \right|  \leq 2 \sqrt{N-1} \mu d_0,
\end{equation*}
which can be made as small as needed to ensure that \eqref{kCB2} follows from \eqref{kCB1_recall} (by taking $d_0$ even smaller if needed).

By choosing $t := R$ in 
\eqref{S_j} and using \eqref{kCB2}, one has
$$
 	S_j(D^2v(x_0)) \geq R^j e^{-jRd_0} \frac{\beta_0}{2} > 0 \ \ \ \text{for each} \ x_0 \in  \Omega_{d_0} \  \text{and} \ j = 1, \ldots k, 
$$
and hence the strict $k$-convexity of $\varphi$ on $\Omega_{d_0}$.

Next we verify the claim \eqref{phi_bdy} concerning boundary values. The claim is trivial on the outer boundary $\partial \Omega$ where $\varphi$ vanishes and $u \geq 0$. On the compact inner boundary (where $d(x) = d_0$), by reducing $d_0$ if need be, we can assume that $u \geq -1/2$ (since $u$ is lower semi-continuous and is non-negative on $\partial \Omega$). Hence it suffices to choose $t > 0$ large enough so that 
$$
\varphi = e^{-td_0} - 1 \leq -1/2 \leq u.
$$
Recall that we may freely increase $t \geq R$ without spoiling the $k$-convexity of $\varphi$ as noted above.

Finally, we need to verify the subsolution claim \eqref{phi_sub} which using the negativity of $\varphi$ is equivalent to
\begin{equation}\label{phi_sub1}
S_k (D^2 \varphi (x_0)) > \lambda |\varphi(x_0)|^k \ \ \text{for each} \ x_0 \in \Omega_{d_0}.
\end{equation}
Using \eqref{S_j} and \eqref{kCB2} with $j=k$, for each $x_0 \in \Omega_{d_0}$, we have
\begin{eqnarray}\label{phi_sub2}
S_k (D^2 \varphi (x_0)) & = & t^k e^{-kt d(x_0)} \sigma_k \left( \frac{\kappa_1(y_0) }{1 - \kappa_1(y_0) d(x_0)} , \ldots, \frac{\kappa_{N-1}(y_0)}{1 - \kappa_{N-1}(y_0)d(x_0)} , t \right) \nonumber \\ & > & t^k e^{-kt d(x_0)} \frac{\gamma_0}{2}.
\end{eqnarray}
Now, on $\Omega_{d_0}$ (where $d(x_0) < d_0$), we have
\begin{equation}\label{phi_sub3}
\lambda |\phi(x_0)|^k = \lambda \left( 1 - e^{-t d(x_0)}\right)^k < \lambda \left( 1 - e^{-t d_0}\right)^k.
\end{equation}
Combining \eqref{phi_sub2} with \eqref{phi_sub3}, we will have \eqref{phi_sub1} provided that 
$$
	t^k e^{-kt d(x_0)} \frac{\gamma_0}{2} > \lambda \left( 1 - e^{-t d_0}\right)^k,
$$
which holds if $d_0$ is chosen small enough.
\end{proof}

The final boundary estimate  we will need is similar to the preceding estimate and will be employed in the proof of the uniform H\"older regularity for a sequence of solutions tending to a principal eigenfunction, see Theorems \ref{thm:GHB} and \ref{thm:psi_1}.

\begin{proposition}\label{prop:be3}
	Let $\Omega$ be a bounded strictly $(k-1)$-convex domain and let $f \geq 0$ be a bounded continuous function on $\Omega$. Suppose that $u \in C(\overline{\Omega})$ is a $k$-convex solution of
$$
	\left\{\begin{array}{lc}
	S_k(D^2u) = f &\mbox{in}\ \Omega\\
	u = 0 &\mbox{on}\ \partial\Omega
	\end{array}
	\right. 
$$
	Then there exist $d_0 >0$ and $C_3 > 0$ such that
	\begin{equation}\label{barrier_estimate3}
	\mbox{  $u(x) \geq -C_3 d(x)$ \ for all $ x \in \Omega_{d_0}$,}
	\end{equation}
	where, as always, $\Omega_{d_0}:= \{ x\in \Omega: \ 0 < d(x) < d_0\}$.
\end{proposition}

\begin{proof} We will show that we can choose $M, t > 0$ large enough and $d_0$ small enough so that
\begin{equation}\label{barrier_estimate4}
	\mbox{  $u(x) \geq -M \log{(1 + t d(x))}$ \ for all $ x \in \Omega_{d_0}$,}
\end{equation}
	from which \eqref{barrier_estimate3} follows by choosing $C_3 \geq Mt$.
	
	To establish \eqref{barrier_estimate4}, we will compare $u$ with $v:=  -M \log{(1 + t d(\cdot))}$ on a suitable tubular neighborhood $\Omega_{d_0} \subset \Omega_{\delta}$ where $\delta > 0$ is the parameter of Lemma \ref{lem:good_tubes} (for which $d \in C^2(\overline{\Omega}_{\delta})$). More precisely, we will choose the constants $M, t$ large enough and $d_0 \leq \delta$ small enough so that $v \in C(\overline{\Omega}_{\delta}) \cap C^2(\overline{\Omega}_{\delta})$ satisfies
\begin{equation}\label{be3_PDE1}
		D^2 v(x_0) \in \Sigma_k^{\circ} \ \ \text{and} \ \ S_k(v(x_0)) > \sup_{\Omega} f  \ \ \text{for each} \ x_0 \in \Omega_{d_0}
\end{equation}
	and 
\begin{equation}\label{be3_BC}
	u \geq v\ \ \text{on} \ \partial \Omega_{d_0}.
\end{equation}
	Since $u$ satisfies (in the  $\Sigma_k$-admissible viscosity sense)
\begin{equation}\label{be3_PDE2}
	S_k(D^2u) = f \leq \sup_{\Omega} f  \ \ \ \text{in} \ \Omega_{d_0} \subseteq \Omega,
\end{equation}
	by combining \eqref{be3_PDE1}-\eqref{be3_PDE2}, the comparison principle of Theorem \ref{thm:cp2} gives the desired conclusion $u \geq v$ on $\Omega_{d_0}$.
	
	We first choose $d_0 \leq \delta$ so that $v$ is strictly $k$-convex in $\Omega_{d_0}$ using the same argument employed in the proof of Proposition \ref{prop:be2}. Briefly, we first calculate to find that for each $x_0 \in \Omega_{\delta}$ and for each $j = 1, \ldots, k$ we have
\begin{equation}\label{S_j_for_v}
	S_j(D^2 v(x_0)) = \left( \frac{Mt}{1 + td} \right)^j \sigma_j \left( \frac{\kappa_1 }{1 - \kappa_1 d} , \ldots, \frac{\kappa_{N-1}}{1 - \kappa_{N-1}d} , \frac{t}{1 + td} \right),
\end{equation}
	with $\kappa_i = \kappa_i(y_0)$ and $d = d(x_0)$ in a principal coordinate system based at $y_0 \in \partial \Omega$ which realizes the distance to $x_0 \in \Omega_{\delta}$. Next, we use the strict $(k-1)$-convexity of $\partial \Omega$ to choose $d_0 \leq \delta$ small enough and $t$ large enough so that $\frac{t}{1 + td_0} \geq R$ (where $R$ is the parameter of Lemma \ref{lem:k-1_convex_bdy}) and for which there is a $\beta_0 > 0$ for which \eqref{kCB2} holds. Hence
\begin{equation}\label{be3_PDE3}
	S_j(D^2 v(x_0)) \geq \left( \frac{Mt}{1 + td_0} \right)^j \frac{\beta_0}{2} > 0,
\end{equation}
	 for each $x_0 \in \Omega_{d_0}$ and each $j = 1, \ldots, k$. This gives the needed $k$-convexity ( the first requirement in \eqref{be3_PDE1}).
	 
	 Next, using \eqref{be3_PDE3} with $j = k$, we find
\begin{equation}\label{be3_PD43}
	S_k(D^2 v(x_0)) \geq  \left( \frac{Mt}{1 + td_0} \right)^k \frac{\beta_0}{2} > \sup_{\Omega} f \ \ \text{for each} \ x_0 \in \Omega_{d_0}, 
\end{equation} 
	by choosing $M > 0$ sufficiently large. This gives the second needed condition in \eqref{be3_PDE1}.
	
	Finally, we verify the needed boundary inequality \eqref{be3_BC} on $\partial \Omega_{d_0} = \partial \Omega \cup \Sigma_{d_0}$ where $\Sigma_{d_0} := \{ x \in \Omega: \ d(x) = d_0\}$. On $\partial \Omega$ one has $u = 0 = v$. For $\Sigma_0$, we need to show
\begin{equation}\label{be3_BC2}
		-v(x) = M \log{(1 + t d_0)} \geq  - u(x) \ \ \text{for each} \ x \in \Sigma_{d_0},	\end{equation}
	where $u \in C(\overline{\Omega})$ is bounded and $u \leq 0$ on $\Omega$ since $u$ is $k$-convex and vanishes on the boundary (see Corollary \ref{cor:PEC}). Clearly, it suffices to choose $M$ so that
	\begin{equation}\label{choseC_3}
		M \geq \frac{1}{\log (1 + t d_0)} \sup_{\Omega} (-u)
	\end{equation}
in order to have \eqref{be3_BC2}.
\end{proof}

\section{Characterization of the principal eigenvalue}

In all that follows, $\Omega \subset \R^N$ will be a bounded open domain with $C^2$ boundary which is strictly $(k-1)$-convex in the sense of Definition \ref{defn:k-1_convex_bdy}. Denote by 
\begin{equation}\label{def:Phi_k}
\Phi_k^{-}(\Omega) := \{ \psi \in \USC(\Omega): \psi \ \text{is $k$-convex and negative in} \ \Omega\},
\end{equation}
where the notion of $k$-convexity is that of Definition \ref{defn:k_convex}. Notice  
that since $\psi$ is bounded from above (by zero) on $\Omega$ one can extend $\psi$ to a $\USC(\overline{\Omega})$ function in a canonical way by letting
\begin{equation}\label{extend_to_boundary}
	\psi(x_0) := \limsup_{\substack{x \to x_0 \\ x \in \Omega}} u(x), \ \ \text{for each} \ x_0 \in \partial \Omega.
\end{equation}
In particular, since $\psi < 0$ on $\Omega$ this extension also satisfies
\begin{equation}\label{extend}
	\psi \leq 0 \ \ \text{on} \ \partial \Omega.
\end{equation}
We will freely make use of this extension so that   Proposition \ref{prop:be1} (Hopf's Lemma) applies to give the boundary estimate \eqref{barrier_estimate1} for the canonical extension of $\psi \in \Phi_k^{-}(\Omega)$. 

The following definition gives a candidate for the principal eigenvalue associated to a negative $k$-convex eigenfunction of the $k$-Hessian $S_k(D^2u) = \sigma_k(\lambda(D^2u))$.
\begin{defn}\label{defn:PEV} For each $k \in \{1, \ldots, N\}$ fixed, define 
\begin{equation} \label{def:peval1}
\lambda_1^{-}(S_k, \Sigma_k) := \sup \Lambda_k
\end{equation}
where 
\begin{equation}\label{def:peval2}
\Lambda_k:= \{ \lambda \in \R: \ \exists \psi \in \Phi_k^{-}(\Omega)\ \text{with} \  S_k(D^2 \psi) + \lambda \psi |\psi|^{k-1} \geq 0 \ \text{in} \ \Omega \}.
\end{equation}
The meaning of the differential inequality in \eqref{def:peval2} is in the viscosity sense; that is,
for each $x_0 \in \Omega$ and for each $\varphi$ which is $C^2$ near $x_0$ one has that
\begin{equation}\label{kc_test1}
\mbox{ $\psi - \varphi$ has a local maximum in $x_0 \ \ \Rightarrow \ \ S_k(D^2 \varphi(x_0)) +\lambda \psi(x_0) |\psi(x_0)|^{k-1} \geq 0$. }
\end{equation}
Moreover, since $\psi \in \Phi_k^-(\Omega)$ is $k$-convex,  $\psi$ is a $\Sigma_k$-admissible subsolution of the PDE in the sense of Definition \ref{defn:sub_super} (a), whose canonical extension to the boundary \eqref{extend_to_boundary} is admissible for Proposition \ref{prop:be1} (as noted above).
\end{defn}

About the definition, a few elementary remarks are in order which we record in the following lemma. 

\begin{lem}\label{lem:PEV} Let $\lambda_1^-(S_k, \Sigma_k)$ and $\Lambda_k$ be as in Definition \ref{defn:PEV}. Then the 
following facts hold.
	\begin{itemize}
		\item[(a)] $(- \infty, \lambda_1^-(S_k, \Sigma_k))\subset \Lambda_k$, or equivalently, if $\lambda < \lambda_1^-(S_k, \Sigma_k)$ then there exists $\psi \in \USC(\Omega)$ which is $k$-convex and negative in $\Omega$ and satisfies
		 \begin{equation}\label{PEVequation}
		S_k(D^2 \psi) + \lambda \psi |\psi|^{k-1} \geq 0 \ \ \text{in} \ \Omega.
		\end{equation}
		\item[(b)] If $\Omega$ (which is bounded) is contained in $B_R(0)$ then one has the estimate
		\begin{equation}\label{PEVest1}
			\lambda_1^-(S_k, \Sigma_k) \geq C_{N,k} R^{-2k} \quad \text{where} \ \ C_{N,k} = \left( \begin{array}{c} N \\ k \end{array} \right) := \frac{N!}{k! (N - k)!}.
		\end{equation} 
		In particular, $\lambda_1^-(S_k,\Sigma_k)$ is positive.    
	\end{itemize}                                                       
\end{lem}

\begin{proof}
	For the part (a), we first claim that if $\lambda \in \Lambda_k$ then $(-\infty, \lambda] \subset \Lambda_k$. By the definition of $\Lambda_k$, there is $\psi \in \Phi^-(\Omega)$
	as defined in \eqref{def:Phi_k} which satisfies \eqref{PEVequation}. If $\lambda^* < \lambda$ then this $\psi \in \Phi^-(\Omega)$ satisfies \eqref{PEVequation} with $\lambda^*$ in place of $\lambda$. Indeed, for each $x_0 \in \Omega$ and each
	$\varphi$ which is $C^2$ near $x_0$ one has \eqref{kc_test1} and hence
$$
S_k(D^2 \varphi(x_0)) +\lambda^* \psi(x_0) |\psi(x_0)|^{k-1} + (\lambda - \lambda^*) \psi(x_0) |\psi(x_0)|^{k-1} \geq 0,
$$
but the second term is negative and hence the claim. 

Now, if $\lambda < \lambda_1^-(S_k, \Sigma_k)$ then by the definition of $\lambda_1^-(S_k, \Sigma_k)$ there must exist $\lambda^* = \lambda + \veps^*$ between $\lambda$ and  $\lambda_1^-(S_k, \Sigma_k)$ which belongs to $\Lambda_k$ and hence $\lambda + \veps \in \Lambda_k$ for each $\veps \in [0, \veps*]$ by the claim proved above, which completes the proof of part (a).

For part (b), consider the convex (and hence $k$-convex) function $\psi(x):= |x|^2 - M$ with $M > R^2$ so that $\psi < 0 $ on $\overline{\Omega}$ if $\Omega \subset B_R(0)$. One has $D^2 \psi(x) = 2I$ for each $x \in \Omega$ and hence
$$
	S_k(D^2 \psi(x)) + \lambda \psi(x) |\psi(x)|^{k-1} = 2^k C_{N,k} - \lambda (2R^2 - |x|^2)^k.
$$
Since $2R^2 - |x|^2 > 0$ by hypothesis, for $\lambda > 0$ one has
$$
S_k(D^2 \psi(x) + \lambda \psi(x) |\psi(x)|^{k-1} \geq 2^k \left[C_{N,k} - \lambda R^{2k} \right] \geq 0,
$$
provided $\lambda R^{2k} \leq C_{N,k}$. The claim \eqref{PEVest1} follows. 
%claim (e) follows from part (d) as $\lambda \in \Lambda_k$.
\end{proof}

We are now ready for the main result of this section, which we state in the nonlinear case i.e. $k > 1$. 

\begin{thm}\label{thm:MPC}
Let $k \in \{2, \ldots, N\}$ and let $\Omega$ be a strictly $(k-1)$-convex domain in $\R^N$. For every $\lambda < \lambda_1^{-}(S_k, \Sigma_k)$ and for every $u \in \LSC(\overline{\Omega})$ which is $\Sigma_k$-admissible supersolution of
\begin{equation}\label{EVsuperEq}
S_k(D^2u) + \lambda u|u|^{k-1} = 0 \ \ {\rm in} \ \Omega,
\end{equation}
one has the following minimum principle 
\begin{equation}\label{EVMinPrinc}
u \geq 0 \ \ {\rm on} \ \partial \Omega \ \ \Rightarrow \ \ u \geq 0 \ \ {\rm in} \  \Omega
\end{equation}
\end{thm}

Before giving the proof, a pair of remarks are in order. 

\begin{rem}\label{rem:lambda_negative} If $\lambda \leq 0$, then the gradient-free equation \eqref{EVsuperEq} is {\em proper elliptic} on the constraint set $\R \times \Sigma_k$ and the maximum/minimum principle for $(\R \times \Sigma_k)$-admissible viscosity subsolutions/supersolutions of \eqref{EVsuperEq} follows from \cite{CHLP19} (see section 11.1). Hence, we will restrict attention to the interesting case
	\begin{equation}\label{lambda_range}
	0 < \lambda < \lambda_1^{-}(S_k, \Sigma_k).
	\end{equation}
	\end{rem}

\begin{proof}[Proof of Theorem \ref{thm:MPC}] We argue by contradiction. Assume that 
	\begin{equation}
	\mbox {there exists $x \in \Omega$ such that $u(x) < 0$}
	\end{equation}
	and so $u \in \LSC(\overline{\Omega})$ will have a negative minimum on $\overline{\Omega}$ in some interior point  $\bar{x} \in \Omega$. Let $\psi \leq 0$ on $\partial \Omega$ be a $\Sigma_k$-admissible subsolution of
\begin{equation}\label{MPC2}
	S_k(D^2 \psi) + \wt{\lambda} \psi |\psi|^{k-1} \geq 0 \ \ \text{in} \ \Omega \ \ \text{for a fixed} \ \wt{\lambda} \in (\lambda, \lambda_1^{-}(S_k, \Sigma_k)).
\end{equation}
We will compare $u$ with $\gamma \psi$ where
	\begin{equation}\label{MPC1}
	\gamma \in \left( 0, \gamma':= \sup_{\Omega} \frac{u}{\psi} \right)
	\end{equation}
is to be suitably chosen and $\psi \in \USC(\overline{\Omega})$ is $k$-convex, negative in $\Omega$. Notice that such values of $\wt{\lambda} > 0$ exist by \eqref{lambda_range} and that such a $\psi$ exists by Lemma \ref{lem:PEV} (a), where we take the canonical $\USC$ extension to the boundary of  \eqref{extend_to_boundary} so that \eqref{extend} also holds for $\psi$.\\ 
Notice also that if $\psi$ solves \eqref{MPC2} then so does $\gamma \psi$. Indeed, for each $x_0 \in \Omega$, if $\gamma \psi - \varphi$ has a local max in $x_0$ then $\psi - \frac{1}{\gamma} \varphi$ with $\gamma > 0$ does too and hence
\begin{equation}\label{MPC3}
\gamma^{-k}	S_k(D^2 \varphi(x_0)) + \wt{\lambda} \psi(x_0) |\psi(x_0)|^{k-1} \geq 0,
\end{equation} 
which gives \eqref{MPC2} for $\gamma \psi$ by multiplying \eqref{MPC3}  by $\gamma^k > 0$.

\noindent{\bf  Step 1:} {\em Show that $\gamma' > 0$ defined in \eqref{MPC1} is finite: that is, one has}
\begin{equation}\label{step1}
	\sup_{\Omega} \frac{u}{\psi} < + \infty.
\end{equation}

\vspace{1ex} We begin by noting that $\psi < 0$ on $\Omega$ and we have assumed that $u$ has a negative minimum at $\bar{x} \in \Omega$ so that the ratio is positive in $\bar{x}$. Near the boundary, we make use of the boundary estimates of Proposition \ref{prop:be1} for $\psi$ and Proposition \ref{prop:be2} for $u$ to say that there exist $C_1, C_2 > 0$ such that
\begin{equation}\label{BE1}
\psi(x)\leq -C_1 \, d(x) \ \text{for all} \ x \in \Omega_{\delta/2}  
\end{equation}
and
\begin{equation}\label{BE2}
u(x) \geq -C_2 \, d(x) \ \text{for all} \ x \in \Omega_{d_0}, 
\end{equation}
where $\delta > 0$ is the parameter of Lemma \ref{lem:good_tubes} defining a good tubular neighborhood of $\partial \Omega \in C^2$ and $d_0 \leq \delta$ depends on $\mu$ (as defined in \eqref{kCB3}, which bounds the absolute values of the principal curvatures of $\partial \Omega$), the monotonicity properties of $\sigma_j$ for $j \leq k$ on $\overline{\Gamma}_k$ and their moduli of continuity. Hence, by picking
$$
	\rho \leq \min \{ d_0, \delta/2 \}
$$
and recalling that $-\psi > 0$ on $\Omega$, we can use both \eqref{BE1} and \eqref{BE2} on $\Omega_{\rho}$ to find 
$$
\frac{u(x)}{\psi(x)} = \frac{-u(x)}{-\psi(x)} \leq \frac{C_2 \, d(x)}{C_1 \, d(x)} = \frac{C_2}{C_1} \ \ \text{for all} \ x \in \Omega_{\rho}; 
$$
that is
\begin{equation}\label{step1A}
\sup_{\Omega_{\rho}} \frac{u}{\psi} \leq \frac{C_2}{C_1} < +\infty.
\end{equation}

Now, on the compact set $K := \overline{\Omega \setminus \Omega_{\rho}}$ where $-\psi \in \LSC(K)$ and positive and $-u \in \USC(K)$ one has the existence of $\wt{C}_1 > 0$ and $\wt{C}_2$ such that
$$
-\psi(x) \geq \wt{C}_1 > 0 \ \ \text{and} \ \ -u(x) \leq \wt{C}_2 \ \ \text{for each} \ x \in K
$$
to find
\begin{equation}\label{step1B}
\sup_{K} \frac{u}{\psi} \leq \frac{\wt{C}_2}{\wt{C}_1} < +\infty.
\end{equation}
Combining \eqref{step1A} with \eqref{step1B} gives the needed \eqref{step1}.

\noindent{\bf  Step 2:} {\em Reduce the proof to showing that there exists $\wt{x} \in \Omega$ such that $u(\wt{x}) < 0$ and}
\begin{equation}\label{x_tilde_ineq1}
	\lambda |u(\wt{x})|^k \geq \gamma^k \wt{\lambda} |\psi(\wt{x})|^k.
	\end{equation}
	
Indeed, recalling that $u(\wt{x}), \psi(\wt{x}) < 0$, $0 < \lambda < \wt{\lambda} < \lambda_1^{-}$ and $\gamma^{\prime} := \sup_{\Omega}(u/\psi)$, from \eqref{x_tilde_ineq1} one finds
	\begin{equation}\label{x_tilde_ineq2}
	\frac{\wt{\lambda}}{\lambda}\gamma^k \leq \left( \frac{ u(\wt{x})}{\psi(\wt{x})} \right)^k \leq  (\gamma')^k.
	\end{equation}
Now, choose the free parameter $\gamma \in (0, \gamma^{\prime})$ to satisfy
	\begin{equation}\label{pick_gamma}
	\gamma > \left( \frac{\lambda}{\wt{\lambda}} \right)^{1/k} \gamma^{\prime},
	\end{equation} 
	which can be done since $\gamma \in (0, \gamma^{\prime})$ and $(\lambda/\wt{\lambda})^{1/k} < 1$. Raising the inequality \eqref{pick_gamma} to power $k$ and multiplying by $\wt{\lambda} > 0$ gives a contradiction to the inequality \eqref{x_tilde_ineq2}. This completes the proof of the theorem, modulo showing that such an $\wt{x}$ exists.
	
\noindent{\bf  Step 3:} {\em Exhibit $\wt{x} \in \Omega$ such that $u(\wt{x}) < 0$ and \eqref{x_tilde_ineq1} holds.}
	
In order to find $\tilde{x} \in \Omega$ such that \eqref{x_tilde_ineq1} holds when comparing $u$ to  $\gamma \psi$, we make use of the classical viscosity device of looking at the maximum values of the family of upper semicontinuous functions defined by doubling variables and with an increasing (in $j \in \N$) quadratic penalization 
\begin{equation}\label{Psi_j}
\Psi_j(x,y):= \gamma \psi(x) - u(y) - \frac{j}{2}|x-y|^2, \ \ (x,y) \in \overline{\Omega} \times \overline{\Omega}, \ j \in \N.
\end{equation}
For simplicity of notation, we will suppress the free parameter $\gamma \in (0, \gamma^{\prime})$ in the notation for $\Psi_j$ (as well in certain $\gamma$-dependent quantities below), thinking of $\gamma \in (0, \gamma^{\prime})$ as arbitrary, but fixed. 
  
First, notice that for each $j \in \N$, $\Psi_j \in \USC(\overline{\Omega} \times \overline{\Omega})$ will have a maximum value 
\begin{equation}\label{Mj}
M_j := \max_{(x,y) \in \overline{\Omega} \times \overline{\Omega} } \Psi_j(x,y) < +\infty.
\end{equation}

\noindent  
{\bf Claim:} {\em For each $j \in \N$, the maximum $M_j$ defined in \eqref{Mj} is positive.}

For each fixed $\gamma \in (0, \gamma^{\prime})$, we will show that
%the function $v \in \USC(\overline{\Omega})$ defined by
\begin{equation}\label{v_function}
	\Psi_j (x,x) := \gamma \psi(x) - u(x).
	\end{equation} 
%Since $\psi \leq 0 \leq u$ on $\partial \Omega$, one has $v \leq 0$ on the boundary. 
must have a positive value in the interior of $\Omega$, and hence the claim. Assume to the contrary that $\Psi_j (x,x) \leq 0$ on $\Omega$; that is, for each $x \in \Omega$, assume that 
$$
\gamma \psi(x) \leq u(x) \ \ \Leftrightarrow \ \ \frac{u(x)}{\psi(x)} \leq \gamma,
$$
since $\psi < 0$ on $\Omega$. This implies that $\gamma^{\prime}$ which is the sup of $u/\psi$ satisfies $\gamma^{\prime}\leq\gamma$.
However this is a contradiction since $\gamma < \gamma^{\prime}$, which completes the claim.

Hence, using the claim, for each fixed $\gamma \in (0, \gamma^{\prime})$, there exists $\bar{x} \in \Omega$ such that
\begin{equation}\label{Mj_bound}
	M_j:= \max_{\overline{\Omega} \times \overline{\Omega}} \Psi_j  \geq \Psi_j(\bar{x}, \bar{x}) = \gamma \psi(\bar{x}) - u(\bar{x}) :=\bar{m} > 0, \ \forall \, j \in \N.
\end{equation}
Notice that the maximum values $M_j$ decrease as $j$ increases so that
\begin{equation}\label{M_infty}
	M_{\infty} := \lim_{j \to + \infty} M_j = \inf_{j \in \N} M_j \geq \bar{m} > 0.
\end{equation}
Using the finiteness of the limit $M_{\infty}$ and the fact that
$u \geq 0$ and $\psi \leq 0 $ on $\partial \Omega$ one has the following standard facts (see, for example, Lemma 3.1 of \cite{CIL92}). 

\begin{lem}\label{lem:m_j}  For each $j \in \N$ consider any pair $(x_j, y_j) \in \overline{\Omega} \times \overline{\Omega}$ such that
\begin{equation}\label{IL5}
	M_j :=\max_{\overline{\Omega} \times \overline{\Omega}} \Psi_j = \Psi_j(x_j, y_j).
\end{equation}
One has
\begin{equation}\label{IL6}
	\lim_{j \to +\infty} j |x_j - y_j|^2 = 0 \ \ \text{and hence} \ \ (x_j - y_j) \to 0 \ \text{as} \ j \to +\infty;
\end{equation}
\begin{equation}\label{IL7}
	(x_j, y_j) \in \Omega \times \Omega \ \ \text{for all} \ j \ \text{sufficiently large};
\end{equation}
and for any accumulation point $\tilde{x}$ of the bounded sequence $\{x_j\}_{j \in \N}$ one has
\begin{equation}\label{IL8}
	0 < \bar{m} \leq M_{\infty} = \gamma \psi(\tilde{x}) - u(\tilde{x}) = \max_{(x,y) \in \overline{\Omega} \times \overline{\Omega}} (\gamma \psi(x) - u(y))
\end{equation}
and hence $\tilde{x} \in \Omega$ since $\gamma \psi, -u \leq 0$ on $\partial \Omega$.
	\end{lem}

\begin{proof} For completeness, we sketch the argument. The claim \eqref{IL6} follows from the fact that for each $j \in \N$ one has
$$
	M_{j/2} \geq M_j(x_j,y_j) = M_j(x_j, y_j) - \frac{j}{4} |x_j - y_j|^2
$$
and hence by \eqref{M_infty}
$$
	0 \leq \frac{j}{4} |x_j - y_j|^2 \leq M_{j/2} - M_j \to 0
$$
Next, by \eqref{Mj}
$$
	M_j = \gamma \psi(x_j) - u(y_j) - \frac{j}{2} |x_j - y_j|^2 \geq \bar{m} > 0 
$$
and hence \eqref{IL6} yields
$$
	\gamma \psi(x_j) - u(y_j) > 0 \ \ \text{for all sufficiently large} \ j.
$$
Hence for large $j$ one has \eqref{IL7} since $\gamma \psi, -u \leq 0$ on $\partial \Omega$ and $x_j - y_j \to 0$ as $j \to + \infty$. Finally, for the claim \eqref{IL8}, if $x_{j_k} \to \tilde{x}$ as $k \to + \infty$, then so does $y_{j_k}$ and using \eqref{IL6} plus the fact that $\gamma \psi, -u \in \USC(\overline{\Omega})$ yields
$$
	M_{\infty} = \limsup_{k \to + \infty} \left(\gamma \psi (x_{j_k}) - u(y_{j_k}\right) \leq \gamma \psi(\tilde{x}) - u(\tilde{x}) = \Psi_{j_k}(\tilde{x}, \tilde{x}) \leq M_{j_k}, \ \ \forall \, k \in \N.
$$
\end{proof}

One can now exhibit $\wt{x} \in \Omega$ such that $u(\wt{x}) < 0$ and \eqref{x_tilde_ineq1} holds. The idea is to apply Ishii's lemma (as given in the discussion of the formulas (3.9) and (3.10) in Crandall-Ishii-Lions \cite{CIL92})
along positive interior (local) maximum points of $\Psi_j$ and using that $\gamma \psi$ and $u$ are viscosity sub and supersolutions in $\Omega$. More precisely, if
$$
	\Psi_j(x,y) := \gamma \psi(x) - u(y) - \frac{j}{2} |x-y|^2 \in \USC(\overline{\Omega} \times \overline{\Omega})
$$
has a local maximum in $(x_j, y_j) \in \Omega \times \Omega$, then by Lemma \ref{lem:m_j}, for large $j$ these local maxima lie in $\Omega \times \Omega$ and by Ishii's lemma there exist $X_j, Y_j \in \Ss(N)$ such that
\begin{equation}\label{Ishii1}
(j(x_j - y_j), X_j) \in {\overline{J}}^{2,+} \gamma \psi(x_j) \ \ \text{and} \ \ (j (x_j - y_j), Y_j) \in {\overline{J}}^{2,-}  u(y_j)
\end{equation}
where
\begin{equation}\label{Ishii2}
	X_j \leq Y_j \ \ \text{in} \ \Ss(N).
\end{equation}
Furthermore, by the last part of Lemma \ref{lem:m_j}, there exists $\wt{x} \in \Omega$ such that, up to a subsequence,
\begin{equation}\label{Ishii3}
	(x_j, y_j) \to \wt{x} \ \ \text{as} \ j \to +\infty.
\end{equation}
Now, since $\gamma \psi$ is $\Sigma_k$-subharmonic ($k$-convex) in $\Omega$, for each $x \in \Omega$ and for every $p \in \R^N$ one has
\begin{equation}\label{psi_k-convex}
	(p,A) \in J^{2,+} \gamma \psi(x) \ \ \Rightarrow A \in \Sigma_k,
\end{equation}
but $\Sigma_k$ is closed and from the first statement of \eqref{Ishii1} it follows that
\begin{equation}\label{X_j_in_Sigma_k}
	X_j \in \Sigma_k.
\end{equation}
By the positivity property \eqref{P}, combining \eqref{Ishii2} and \eqref{X_j_in_Sigma_k} yields
\begin{equation}\label{Y_j_in_Sigma_k}
	Y_j \in \Sigma_k.
\end{equation}
We remark that this is the key observation that indicates why Ishii's lemma continues to be useful in the case of viscosity solutions with admissibility constraints satisfying the positivity property \eqref{P}.

Next, using that $\gamma \psi$ and $u$ are $\Sigma_k$-admissible subsolutions and supersolutions of \eqref{MPC2} and  \eqref{EVsuperEq} respectively, one has for all large $j$
\begin{equation}\label{psi_sub}
	S_k(X_j) + \wt{\lambda} \gamma^k \psi(x_j) |\psi(x_j)|^{k-1} \geq 0
\end{equation}
and
\begin{equation}\label{u_super}
	S_k(Y_j) + \lambda  u(y_j) u(y_j)|^{k-1} \leq 0
\end{equation}
where we have used the fact that $Y_j \in \Sigma_k$ in the supersolution definition (see Definition \ref{defn:sub_super} (b)). Combining \eqref{psi_sub} and \eqref{u_super} and using the monotonicity of $S_k$ on $\Sigma_k$ (see property \eqref{DE}) allowed by \eqref{X_j_in_Sigma_k} and \eqref{Ishii2} one has for all large $j \in \N$
\begin{equation}\label{MPConclusion1}
	\wt{\lambda} \gamma^k |\psi(x_j)|^k \leq S_k(X_j) \leq S_k(Y_j) \leq \lambda|u(y_j)|^k
\end{equation}
where we have used that $\psi(x_j), u(y_j) < 0$ for large $j$. Now since $-\psi > 0$ is $\LSC(\Omega)$ and $-u \in \USC(\Omega)$, from \eqref{MPConclusion1} one finds
\begin{equation}\label{MPConclusion2}
	0 < \wt{\lambda}(-\psi(\wt{x}))^k \leq  \liminf_{j \to +\infty} \leq  \wt{\lambda} \gamma^k (- \psi (x_j))^k  \leq \limsup_{j \to +\infty} \lambda (-u(y_j))^k \leq \lambda(-u(\wt{x}))^k,
\end{equation}
which gives the needed inequality \eqref{x_tilde_ineq1}. Moreover, as noted in \eqref{u_positive}, one has $u(y_j) < 0$. Since $u$ is $\LSC(\Omega)$, one then has $u(\wt{x}) \leq 0$, but it cannot vanish by \eqref{MPConclusion2}. Thus $u(\wt{x}) < 0$ as needed.
\end{proof}

An immediate consequence of the minimum principle are the
 following characterizations of the principal eigenvalue of $S_k$ discussed by Wang \cite{Wa95} and Lions \cite{PLL85} (in the case $k = N$) using the variational structure of $S_k$. See also Jacobsen \cite{Ja99} for a bifurcation approach.

\begin{cor}\label{cor:PEC} Let $\Omega$ be as in Theorem \ref{thm:MPC} and let $k \geq 2$. Then $\lambda_1^{-}(S_k, \Sigma_k)$ as defined by \eqref{def:peval1}-\eqref{def:peval2} is equal to $\lambda_1^{(k)}$ defined by
\begin{equation}\label{RRFk}
    \lambda_1^{(k)} := \inf_{u \in \Phi_0^k(\Omega)} \left\{ - \int_{\Omega} u S_k(D^2u) \, dx: \ \ ||u||_{L^{k + 1}(\Omega)} = 1 \right\},
\end{equation}
where $\Phi_0^k(\Omega) = \{ u \in C^2(\Omega): S_k(D^2 u) \in \Gamma_k \ \ {\rm and} \ \ u = 0 \ {\rm on} \ \partial \Omega \}$ and $\Gamma_k$ is the open cone \eqref{gamma_k}.
When $k = N$, one has
\begin{equation}\label{RRFN}
   \mbox{ $\lambda_1^{(N)} := \inf\{ \lambda_1^a: \ a \in C(\overline{\Omega}, \Ss(N)) \ \text{such that} \ \ a > 0 , {\rm det}\, a \geq N^{-N} \ \text{in} \ \overline{\Omega}, \}$}
\end{equation}
and $\lambda_1^a$ is the first eigenvalue of the uniformly elliptic operator $- \sum_{i,j = 1}^N a_{ij} D_{ij}$.
\end{cor}

\begin{proof}
Since there exists a $k$-convex principal eigenfunction $\psi_1$ which is negative in the interior and vanishes on the boundary, by the definition of $\lambda_1^{-}(S_k, \Sigma_k)$, one has $\lambda_1^{(k)} \leq \lambda_1^{-}(S_k, \Sigma_k)$. If $\lambda_1^{(k)} < \lambda_1^{-}(S_k, \Sigma_k)$, then $\psi_1$ would be a $\Sigma_k$-admissible supersolution of \eqref{EVsuperEq} with $\lambda = \lambda_1^{(k)}$ and hence $\psi_1 \geq 0$ in $\Omega$ by the minimum principle \eqref{EVsuperEq}, which is absurd.  
\end{proof}

\section{Existence of the principal eigenfunction by maximum principle methods}

Even though Corollary \ref{cor:PEC} shows that a negative principal eigenfunction $\psi_1$ exists for $\lambda_1^{-} = \lambda_1^{-}(S_k, \Sigma_k)$, in order to illustrate a general method which should apply also to non variational perturbations of $S_k$, we will give an alternative proof of the existence of $\psi_1$ by maximum principle methods for $\Sigma_k$-admissible viscosity solutions. 

Since the complete argument to solve \eqref{EVP} is somewhat involved, perhaps it is worth giving the general idea first. We will show that $\psi_1 \in C(\overline{\Omega})$ is the limit as $n \to +\infty$ (up to an extracted subsequence) of the normalized solutions
$$
	w_n := \frac{v_n}{||v_n||_{\infty}}
$$
where each  $v_n \in C(\overline{\Omega})$ is a $\Sigma_k$-admissible viscosity solution of the auxiliary problem
\begin{equation}\label{EVP_vn}
	\left\{ \begin{array}{cc}
	S_k(D^2 v_n) = 1 - \lambda_n v_{n} |v_{n}|^{k-1}  & {\rm in} \ \Omega \\
	v_n = 0 & {\rm on} \ \partial \Omega
	\end{array} \right.
\end{equation}
and $\{ \lambda_n \}_{n \in \N}$ is any fixed sequence of spectral parameters with $0 < \lambda_n \nearrow \lambda_1^{-}$ as $n \to + \infty$. The existence of the solutions $v_n$ to \eqref{EVP_vn} presents the same difficulties as mentioned above for \eqref{EVP}, but for each fixed $\lambda \in (0, \lambda_1^-)$ we will show that the problem
\begin{equation}\label{EVP2}
	\left\{ \begin{array}{cc}
	S_k(D^2 u) = 1 - \lambda u |u|^{k-1}  & {\rm in} \ \Omega \\
	u = 0 & {\rm on} \ \partial \Omega
	\end{array} \right.
\end{equation}
has a $\Sigma_k$-admissible solution $u \in  C(\overline{\Omega})$ by an inductive procedure starting from $u_0 = 0$ and then solving
\begin{equation}\label{EVP_n}
\left\{ \begin{array}{cc}
S_k(D^2 u_n) = 1 - \lambda u_{n-1} |u_{n-1}|^{k-1} := f_n & {\rm in} \ \Omega \\
u_n = 0 & {\rm on} \ \partial \Omega
\end{array} \right.
\end{equation}
for a decreasing sequence of $\Sigma_k$-admissible solutions $\{u_n\}_{n \in \N} \subset C(\overline{\Omega})$ (which are negative in $\Omega$) and then pass to the limit as $n \to +\infty$. Notice that the equation in \eqref{EVP_n} is proper as $u_n$ does not appear explicitly and hence the equation is non increasing in $u_n$. Moreover, it will turn out that one can pass to the limit along a subsequence provided that there is a uniform H\"{o}lder bound on $||u_n||_{C^{0,\alpha}(\overline{\Omega})}$ for each $n \in \N$ and some $\alpha \in (0,1]$.

We begin with the following existence and uniqueness result for the underlying degenerate elliptic Dirichlet problems in \eqref{EVP_n} in the nonlinear 
case $k \in \{ 2, \ldots, N\}$. While this result is not new, for completeness we prefer to discuss it. 

\begin{thm}\label{thm:EU} Let $\Omega$ be a strictly $(k-1)$-convex domain of class $C^2$ and let $f \in C(\overline{\Omega})$ be a nonnegative function. There exists a unique $k$-convex solution $u \in C(\overline{\Omega})$ of the Dirichlet problem
\begin{equation}\label{DPf}
	\left\{ \begin{array}{cc}
	S_k(D^2 u) = f & {\rm in} \ \Omega \\
	u = 0 & {\rm on} \ \partial \Omega.
	\end{array} \right.
\end{equation}
More precisely, there is $\Sigma_k$-admissible solution $u \in C(\overline{\Omega})$ of $S_k(D^2 u) - f(x) = 0$ in $\Omega$ in the sense of Definition \ref{defn:sub_super}(c) such that $u = 0$ on $\partial \Omega$.
\end{thm}

\begin{proof} 

The existence and uniqueness for $\Sigma_k$-admissible viscosity solutions follows  from the main results in \cite{CP17}.  See Theorem 1.2 as applied in section 5 of that paper. When $f > 0$, one has smooth solutions if $\partial \Omega$ is smooth as follows from \cite{CNS85}. 

Briefly, we give an idea  of the proof for completeness sake. 
A $\Sigma_k$-admissible viscosity solution of \eqref{DPf} is a $\Theta_k$-harmonic function which vanishes on the boundary where $\Theta_k: \Omega \to \Ss(N)$ is the {\em uniformly continuous elliptic map} defined by
\begin{equation}\label{UCEM}
	\Theta_k(x):= \{ A \in \Sigma_k: \ F_k(x,A):= S_k(A) - f(x) \geq 0\} \ \ \text{for each} \ x \in \Omega.
\end{equation}
The uniform continuity is with respect to the Hausdorff distance on $\Ss(N)$ and follows from the uniform continuity of $f \in C(\overline{\Omega})$. Using Propositions 5.1 and 5.3 of \cite{CP17}, one has the equivalence between $u \in C(\Omega)$ being  $\Theta_k$-harmonic and $u$ being a $\Sigma_k$-admissible viscosity solution of $F_k(x, D^2u) = 0$ since one can easily verify the needed structural conditions ((1.14)-(1.16) and (1.18)); that is,
\begin{equation}\label{str1}
	F_k(x,A + P) \geq F_k(x,A) \ \ \text{for each} \ x \in \Omega, A \in \Sigma_k, P \in \cP;
\end{equation}
\begin{equation}\label{str2}
	\text{for each} \ x \in \Omega \ \text{there exists} \ A \in \Sigma_k \ \text{such that} \ F_k(x,A) = 0; 
\end{equation}
\begin{equation}\label{str3}
	\partial \Sigma_k \subset \{A \in \Ss(N): \ F_k(x,a) \leq 0\} \ \text{for each} \  x \in \Omega;
\end{equation}
and
\begin{equation}\label{str4}
	F_k(x,A) > 0 \ \ \text{for each} \ x \in \Omega \ \text{and each} \ A \in \Theta_k(x)^{\circ}.
\end{equation}
Conditions \eqref{str1} - \eqref{str3} say that $\Theta_k$ defined by \eqref{UCEM} is an {\em elliptic branch} of the equation $F_k(x, D^2u) = 0$ in the sense of Kyrlov \cite{Kv95} (see Proposition 5.1 of \cite{CP17}) and the non-degeneracy condition \eqref{str4} ensures that $\Theta_k$-superharmonics are $\Sigma_k$-admissible viscosity supersolutions of $F_k(x, D^2u) = 0$ (see Proposition 5.3 of \cite{CP17}). Finally the existence of a unique $u \in C(\overline{\Omega})$ which is $\Theta_k$-harmonic taking on the continuous boundary value $\varphi \equiv 0$ follows from Perron's method (Theorem 1.2 of \cite{CP17}) since $\Theta_k$ is uniformly continuous and the strict $(k-1)$-convexity implies the needed strict $\EC{\Sigma}_k$ and $\EC{\widetilde{\Sigma}}_k$ convexity (which is the content of Proposition \ref{prop:EC_convexity}). 
\end{proof}

\begin{rem}\label{rem:HL} Using the language of Harvey-Lawson \cite{HL18}, one could also say that $(S_k, \Sigma_k)$ is a {\em compatible operator-subequation pair} (see Definition 2.4 of \cite{HL18}) and since the continuous boundary data $\varphi \equiv 0$ has its values in $S_k(\Sigma_k)$, the result follows also from Theorem 2.7 of \cite{HL18}.
	\end{rem}

Next we discuss the global H\"{o}lder regularity of the unique solution to Theorem \ref{thm:EU} in the case $k > N/2$, which will lead to compactness for bounded sequences of solution.

\begin{thm}\label{thm:GHB} Under the assumptions of Theorem \ref{thm:EU}, if $k > N/2$ then the unique solution $u$ to the Dirichlet problem \eqref{DPf} belongs to $C^{0, \alpha}(\overline{\Omega})$ with $\alpha := 2 - N/k > 0$. In particular, there exists $C > 0$ which depends on $\Omega, \alpha$ and $\sup_{\Omega}(-u)$  such that
\begin{equation}\label{GHB1}
	|u(x) - u(y)| \leq C \, |x - y|^{\alpha}, \ \ \forall \ x, y \in \overline{\Omega}.
\end{equation}
\end{thm}
Before giving the proof, we formalize a few observations concerning the restriction $k > N/2$ in the statement.

\begin{rem}\label{rem:GHB} For the proof of the global H\"{o}lder bound \eqref{GHB1}, we will adapt the technique developed in the celebrated paper of Ishii and Lions \cite{IL90}. The key step involves a uniform local interior estimate which uses a comparison principle argument  for the solution $u$ (which is $\Sigma_k$-subharmonic since $f \geq 0$) and a family of comparison functions defined in terms of the auxiliary function $\phi(x):=|x|^\alpha$, where $\alpha \in (0,1]$. One needs that $\phi$ is $\Sigma_k$-superharmonic on its domain. It is known that for $\alpha=2-\frac{N}{k}$, the function $\phi$ is a classical $\Sigma_k$-harmonic away from the origin, but $\alpha > 0$ requires the condition $k > N/2$. This restriction can be interpreted in terms of the {\em Riesz characteristic} of the closed convex cone $\Sigma_k \subset \Ss(N)$ as described in Harvey-Lawson \cite{HL18a}. Using the measure theoretic techniques developed by Trudinger and Wang \cite{TW97, TW99} and Labutin \cite{La02}, perhaps it is possible to obtain the global H\"older bound \eqref{GHB1} if $k \leq N/2$. However, our intended focus is limited to maximum principle techniques and hence we have not pursue such improvements here.
\end{rem}
\begin{proof}[Proof of Theorem \ref{thm:GHB}] Since $u \in C(\overline{\Omega})$ by Theorem \ref{thm:EU}, the claim that 
$u \in C^{0, \alpha}(\overline{\Omega})$, reduces to proving the estimate \eqref{GHB1}. Notice that $u$ is $k$-convex (it is a $\Sigma_k$-admissible subsolution) and $u$ vanishes on the boundary and hence $u \leq 0$ in $\Omega$ by Corollary \ref{cor:cp1}. In addition, $u < 0$ in $\Omega$ by Theorem \ref{thm:SMP}. In particular,
\begin{equation}\label{sup_norm}
	||u||_{\infty} := \sup_{\Omega}|u| = \sup_{\Omega}(-u).
\end{equation}
For the H\"older estimate \eqref{GHB1}, it suffices to find $\rho > 0$ and  $C_{\rho} > 0$ for which 
\begin{equation}\label{GHB2}
|u(x) - u(y)| \leq C_{\rho} \, |x - y|^{\alpha}, \ \ \forall \ x, y \in \overline{\Omega} \ \ \text{with} \ |x - y| < \rho.
\end{equation}
In fact, as is well known, if \eqref{GHB2} holds, then using exploiting the boundedness of $u$ yields
$$
\sup_{\substack{x, y \in \overline{\Omega} \\ x \neq y}} \frac{|u(x) - u(y)|}{|x-y|^{\alpha}} \leq \max \left\{ \frac{2 ||u||_{\infty}}{\rho^{\alpha}}, C_{\rho} \right\}.
$$

In order to prove \eqref{GHB2}, first consider the case when $y$ lies on $\partial \Omega$ (the argument for $x \in \partial \Omega$ is the same). In this case, $u(y) = 0$ and $u(x) \leq 0$. Then the boundary estimate of Proposition \ref{prop:be3} shows that $y \in \partial \Omega$ and each $x \in \Omega_{d_0}$ one has  
\begin{equation}\label{GHB3}
|u(y) - u(x)|  = -u(x) \leq c_3 \, d(x)  = C_3 \, \min_{z \in \partial \Omega} |x-z| \leq C_3 |x-y|. 
\end{equation}
By choosing
\begin{equation}\label{GHB4}
	\rho \leq \min \{ d_0, 1 \} \ \ \text{and} \ \ C_{\rho} \geq C_3
\end{equation}
one has \eqref{GHB2} for each $\alpha \in (0,1]$ if $y$ (or $x$) lies on the boundary.

Next, let $y \in \Omega$ and consider the comparison function 
\begin{equation}\label{v_y}
	v_y(x):= u(y) + C_{\rho} \, |x-y|^{\alpha}
\end{equation}
One wants determine $\rho > 0$ sufficiently small and $C_{\rho} > 0$ sufficiently large (recall the restrictions \eqref{GHB4}) so that 
\begin{equation}\label{GHB5}
	u(x) \leq v_y(x) \ \ \text{for each} \ x \in \Omega \cap B_{\rho}(y)
\end{equation}
and hence
\begin{equation}\label{GHB6}
u(x) - u(y) \leq C_{\rho} \, |x-y|^{\alpha} \ \ \text{for each} \ x \in \Omega \cap B_{\rho}(y).
\end{equation}
Then, by exchanging the roles of $x$ and $y$, one would have
\begin{equation}\label{GHB7}
|u(x) - u(y)| \leq C_{\rho} \, |x-y|^{\alpha} \ \ \text{for each} \ x,y \in \Omega \ \text{with} \ |x-y|< \rho,
\end{equation}
which would then complete the proof.

In order to establish \eqref{GHB5}, notice that $u$ is a $\Sigma_k$-admissible solution of $S_k(D^2u) = f \geq 0$ in $\Omega$. In particular, $u$ is $\Sigma_k$-subharmonic ($k$-convex) in $\Omega$. Moreover, for $k > N/2$ one knows that for each $y \in \R^N$, the function defined by
\begin{equation}\label{fundamental solution}
	w_k(x):= |x -y|^{2-N/k} \ \ \text{for} \ x \in \R^N \setminus \{y\}
	\end{equation}
is smooth, $k$-convex and satisfies $S_k(w_k) \equiv 0$ on its domain (see section 2 of  \cite{TW99}). The same is obviously true for the translated version $v_y$ of \eqref{v_y} with the choice $\alpha:= 2 - N/k$ when $k > N/2$. Indeed, using the radial formula \eqref{S_k_radial1} of Lemma \ref{lem:radial} with $h(r):= u(y) + C_{\rho} \, r^{\alpha}$ one finds that 
\begin{equation}\label{FSCalc}
	S_j(D^2v_y(x)) = \left( C_{\rho} \, \alpha |x - y|^{\alpha - 2} \right)^j
	\frac{(N -1)!}{j! \, (N-j)!} \left[ (\alpha - 2)j + N \right] \ \ \text{for each} \ x \neq y.
\end{equation}
When $\alpha = 2 - N/k$,  this is positive for every $j = 1, \ldots, k-1$ and it vanishes for $j=k$. In particular, $v_y$ is $\Sigma_k$-superharmonic in every punctured ball $\dot{B}_{\rho}(y) = B_{\rho}(y) \setminus \{y\}$. Hence, by the comparison principle (Theorem \ref{thm:cp1}) for $\Sigma_k$ sub and superharmonics we will have
\begin{equation}\label{cp:uv}
	u(x) \leq v_y(x) \ \ \text{for each} \ x \in \Omega \cap \dot{B}_{\rho}(y)
\end{equation}
provided that
\begin{equation}\label{cp:boundary}
	u \leq v_y \ \ \text{on} \ \partial(\Omega \cap \dot{B}_{\rho}(y))
\end{equation}
where $\partial(\Omega \cap \dot{B}_{\rho}(y)) = \{y\} \cup (\partial B_{\rho}(y) \cap \Omega) \cup (\partial \Omega \cap \overline{B}_{\rho}(y))$. 

We analyze the three possibilities. At the point $y$, one has
$$
	u(y) = v_y(y) = u(y) + C_{\rho} \, |y - y|.
$$
Next, for $x \in \partial \Omega \cap \overline{B}_{\rho}(y)$ (which is empty if $B_{\rho}(y) \subset \subset \Omega$) one has
$$
u(x) = 0 \ \ \text{while} \ v_y(x) = u(y) + C_{\rho} \, |x-y|^{\alpha},
$$
where $u(y) < 0$ for $y \in \Omega$ as noted above. Since $x \in \partial \Omega$, with $\rho \leq  d_0$ the condition $|x-y| < \rho$ means that $y \in \Omega_{d_0}$ and one can again use the boundary estimate of Proposition \ref{prop:be3} to estimate $u(y)$ from below 
$$
	v_y(x) \geq  -C_3 \, |x-y| + C_{\rho} \, |x-y|^{\alpha} \geq (C_{\rho} - C_3)|x - y|^{\alpha} \geq 0,
$$ 
provided that $\rho \leq 1$ and $C_{\rho} \geq C_3$ as in \eqref{GHB4}.
Finally, if $x \in \partial \Omega \cap \overline{B}_{\rho}(y)$ we will have $v_y(x) = u(y) + C_{\rho} \, |x-y|^{\alpha} \geq u(x)$ if
$$
	|u(x) - u(y)| \leq C_{\rho} \, \rho^{\alpha}.
$$
Having now fixed $\rho \leq \min \{d_0, 1\}$, since $|u(x) - u(y)| \leq 2 ||u||_{\infty}$ it is enough to choose 
\begin{equation}\label{cp:end}
	C_{\rho} \geq \frac{2 ||u||_{\infty}}{\rho^{\alpha}},
\end{equation}
in addition to $C_{\rho} \geq C_3$.
\end{proof}

We now implement the iteration scheme (sketched above) to prove 
 the existence of a principal eigenfunction $\psi_1$ associated to $\lambda_1^{-} = \lambda_1^{-}(S_k, \Sigma_k)$ in the  ``regular case'' with $k > N/2$.

\begin{rem}\label{rem:UC}
	We will make use of the fact that the sets of $\Sigma_k$-subharmonic and $\Sigma_k$-superharmonic functions on $\Omega$ are closed under the operation of taking uniform limits in $\Omega$ of sequences. See property (5)' in \cite{HL09}, for example. 
\end{rem}

\begin{thm}\label{thm:psi_1} Suppose that $k > N/2$. Let $\Omega$ be a strictly $(k-1)$-convex domain of class $C^2$. If $\{v_n\}_{n \in \N}$ is the sequence of $k$-convex solutions \eqref{EVP_vn} with $0 < \lambda_n \nearrow \lambda_1^{-}$ as $n \to + \infty$, then the normalized sequence defined by $w_n:= v_n/||v_n||_{\infty}$ admits a subsequence which converges uniformly to a principal eigenfunction $\psi_1$ for \eqref{EVP}, which is negative on $\Omega$.
\end{thm}

\begin{proof} We divide the proof into two big steps, with several claims to be justified.
	
\noindent{\bf  Step 1:} {\em For each $\lambda \in (0, \lambda_1^-)$, show that there exists a $\Sigma_k$-admissible solution $u$ to the Dirichlet problem \eqref{EVP2}; that is,}
\begin{equation}\label{EVP2}
	\left\{ \begin{array}{cc}
	S_k(D^2 u) = 1 - \lambda u |u|^{k-1}  & {\rm in} \ \Omega \\
	u = 0 & {\rm on} \ \partial \Omega
	\end{array} \right.
\end{equation}

As indicated above, we will look for $u$ as a decreasing limit of solutions $\{u_n\}_{n \in \N_0}$ of the Dirichlet problem  \eqref{EVP_n}, that is,
\begin{equation}\label{EVP_un}
	\left\{ \begin{array}{cc}
	S_k(D^2 u_n) = 1 - \lambda u_{n-1} |u_{n-1}|^{k-1} := f_n & {\rm in} \ \Omega \\
	u_n = 0 & {\rm on} \ \partial \Omega
\end{array} \right.
\end{equation}
With $u_0 \equiv 0$, we apply Theorem \ref{thm:EU} to find $u_1 \in C(\overline{\Omega})$ a $\Sigma_k$-admissible solution of
$$
	\mbox{	$S_k(D^2u_1) = 1$ \ \ in \ $\Omega$ \ \ \ \ \and \ \ \ \ $u_1 = 0$ \ on $\partial \Omega$.}
$$
Since $u_1$ is a $\Sigma_k$-admissible solution it is necessarily $k$-convex and hence satisfies the (strong) maximum principle so that $u_1 \leq 0$ on $\overline{\Omega}$ (and $u_1 < 0$ in $\Omega$) and hence
$$
f_2:= 1 - \lambda u_1 |u_1|^{k-1} = 1 + \lambda |u_1|^k \geq 0
$$
and the induction proceeds using Theorem \ref{thm:EU} to produce the sequence $\{u_n\}_{n \in \N_0}$ of non-positive $\Sigma_k$-admissible solutions which also satisfy
\begin{equation}\label{SMP_un}
	u_n < 0 \ \ \text{on} \ \Omega \ \text{for each} \ n \in \N.
\end{equation}

\noindent  
{\bf Claim 1:} {\em  $\{u_n\}_{n  \in \N_0}$ is a decreasing sequence of $k$-convex functions.}

	By construction, all of the functions vanish on $\partial \Omega$ and are negative in $\Omega$ for $n \geq 1$. We use induction. As we have seen $u_1 < 0 := u_0$ on $\Omega$.  Assuming that $u_{n-1} \leq u_n$ on $\Omega$, we need to show that $u_{n+1} \leq u_n$ on $\Omega$. We have that $u_{n+1}$ is a $\Sigma_k$-admissible solution of
$$
	S_kD^2(u_{n+1}) = 1 - \lambda u_n |u_n|^{k-1} = 1 + \lambda |u_n|^{k} \ \ \text{in} \ \Omega,
$$
	where we have again used $u_n \leq 0$, but then the inductive hypothesis yields
\begin{equation}\label{un_decreasing}
	u_{n-1} \leq u_n \leq 0 \ \ \Rightarrow \ \ |u_n| = -u_n \geq -u_{n-1} = |u_{n-1}| 
\end{equation}
	and hence
$$
	S_kD^2(u_{n+1}) \geq 1 + \lambda |u_{n-1}|^k = f_n = S_k(D^2u_n).
$$
	By the comparison principle, one concludes that $u_{n+1} \leq u_n$ on $\Omega$.
	
\noindent  
{\bf Claim 2:} {\em The sequence  $\{u_n\}_{n  \in \N}$ is bounded in sup norm; that is, there exists $M > 0$ and finite such that}
\begin{equation}\label{un_bounded}
	||u_n||_{\infty} = \sup_{\Omega}(-u_n) \leq M \ \ \emph{for each} \ n \in \N.
\end{equation}
	
	We argue by contradiction, assuming that the increasing sequence $||u_n||_{\infty}$ satisfies $\lim_{n \to +\infty} ||u_n||_{\infty} = +\infty$. Since $u_n < 0$ on $\Omega$ for each $n \in \N$, we can define
\begin{equation}\label{define_vn}
	v_n:= \frac{u_n}{||u_n||_{\infty}} \ \ \text{so that} \ \ ||v_n||_{\infty} = 1 \ \ \text{for each} \ n \in \N.
\end{equation}
	Since the equation $S_k(D^2u_n) = 1 + \lambda |u_{n-1}|^k$ is homogeneous of degree $k$, one has
\begin{equation}\label{vn_eqn1}
	S_k(D^2v_n) = \frac{1}{||u_n||_{\infty}^k}  + \lambda \frac{  ||u_{n-1}||_{\infty}^k}{||u_n||_{\infty}^k} \, |v_{n-1}|^k.
\end{equation}
	Now, making again use of the negativity and monotonicity in \eqref{un_decreasing} one has
\begin{equation}\label{define_betan}
		\beta_n:= \frac{||u_{n-1}||_{\infty}^k}{||u_n||_{\infty}^k} \in (0, 1]
\end{equation}
and combining \eqref{vn_eqn1} with \eqref{define_betan} yields
\begin{equation}\label{vn_eqn2}
	S_k(D^2 v_n) =  \frac{1}{||u_n||_{\infty}^k}  + \lambda \beta_n |v_{n-1}|^k := g_n,
\end{equation}	
	where $g_n\in C(\overline{\Omega})$ and is non-negative. Since $k > N/2$ and since the global H\"older bound of Theorem \ref{thm:GHB} depends only on $\Omega, \alpha$ and $||v_n||_{\infty} \equiv 1$, the sequence of solutions $\{v_n\}_{n \in \N}$ is bounded in  $C^{0, 2 - N/k}(\overline{\Omega})$ and hence admits $v \in C(\overline{\Omega})$ and a subsequence such that
\begin{equation}\label{uniform_limit}
	v_{n_j} \to v \ \ \text{uniformly on} \ \overline{\Omega}.
\end{equation}
In addition,  $0 < \beta_{n_j} \leq 1$ is increasing so converges to some $\beta_{\infty} \in (0, 1]$. The uniform limit $v$ is a $\Sigma_k$-admissible (super)solution of
$$
	\mbox{	$S_k(D^2 v) + \lambda \beta_{\infty} v |v|^{k-1} = 0$ \ \ in \ $\Omega$ \ \ \ \ \and \ \ \ \ $v = 0$ \ on $\partial \Omega$,}
$$
where $\lambda \beta_{\infty} \leq \lambda < \lambda_1^-$. By the minimum principle characterization of $\lambda_1^-$, we must have $v \geq 0$ in $\Omega$. However, each $u_n \in C(\overline{\Omega})$ is negative in $\Omega$ and hence has a negative minimum at some interior point $x_n \in \Omega$ and hence $v_n(x_n) = -1$ for each $n$ which contradicts the fact that uniform limit of \eqref{uniform_limit} satisfies $v \geq 0$ on $\Omega$.

\noindent  
{\bf Claim 3:} {\em The sequence $\{ u_n \}_{n \in \N}$ admits a subsequence $\{u_{n_j}\}_{j \in \N} \subset C(\overline{\Omega})$ which converges uniformly on $\overline{\Omega}$ to a $\Sigma_k$-admissible solution $u$ of \eqref{EVP2}.}
	
	Exploiting the boundedness of Claim 3 for the sequence $\{ u_n \}_{n \in \N}$, we can use the same argument involving the global H\"older estimate of Theorem \ref{thm:GHB} to  extract a uniformly convergent subsequence with limit $u \in C(\overline{\Omega})$ with limit $u$ which is a $\Sigma_k$-admissible solution of \eqref{EVP2}. This completes Step 1 of the proof.
	
\noindent{\bf  Step 2:} {\em Show that there exists $\psi_1 \in C(\overline{\Omega})$ which is negative in $\Omega$ and is a $\Sigma_k$-admissible solution of \eqref{EVP}; that is, }	
\begin{equation}\label{EVP_recall}
	\left\{ \begin{array}{cc}
	S_k(D^2 \psi_1) + \lambda_1^{-}\psi_1 |\psi_1|^{k-1} = 0 & {\rm in} \ \Omega \\
	\psi_1 = 0 & {\rm on} \ \partial \Omega
	\end{array} \right.
\end{equation}
	
	Consider a sequence $\{ \lambda_n\}_{n \in \N} \in (0, \lambda_1^-)$ with $\lambda_n \nearrow \lambda_1^-$ and the associated sequence $\{ v_n\}_{n \in \N} \subset C(\overline{\Omega})$ of solutions to \eqref{EVP2} with $\lambda = \lambda_n$; that is, 
\begin{equation}\label{EVP2n}
	\left\{ \begin{array}{cc}
	S_k(D^2 v_n) = 1 - \lambda_n v_n |v_n|^{k-1}  & {\rm in} \ \Omega \\
	v_n = 0 & {\rm on} \ \partial \Omega
	\end{array} \right.
\end{equation}
	
	 Since each $v_n$ is $\Sigma_k$-subharmonic in $\Omega$ and vanishes on the boundary, $v_n < 0$ on $\Omega$ for each $n$.
	 
\noindent  
{\bf Claim 4:} {\em One has $||v_n||_{\infty} \to + \infty$ as $n \to +\infty$.}

We argue by contradiction. If not, then again by the global H\"older bound of Theorem \ref{thm:GHB} we can extract a subsequence of these functions which are $\Sigma_k$-subharmonic and negative in $\Omega$ and  which converges uniformly on $\overline{\Omega}$ to a $\Sigma_k$-admissible solution $w \in C(\overline{\Omega})$ to the Dirichlet problem 	 
\begin{equation}\label{EVPw}
	\left\{ \begin{array}{cc}
	S_k(D^2 w) +   \lambda_1^- w |w|^{k-1} = 1 & {\rm in} \ \Omega \\
	w = 0 & {\rm on} \ \partial \Omega.
	\end{array} \right.
\end{equation}
Since $w \in C(\overline{\Omega})$ is non-positive, there exists $\veps > 0$ such that
\begin{equation}\label{w_bound}
	-\veps w |w|^{k-1} \leq 1 \ \ \text{in} \ \overline{\Omega}.
\end{equation}
Hence $w$ is a $k$-convex, negative in $\Omega$ and satisfies (in the $\Sigma_k$-admissible viscosity sense)  
\begin{equation}\label{subsolution}
	S_k(D^2 w) + (\lambda_1^- + \veps) w |w|^{k-1} \geq 0,
\end{equation}
which contradicts the Definition \ref{defn:PEV} of $\lambda_1^-(S_k, \Sigma_k)$. Hence Claim 4 holds.

Finally, consider the normalized sequence defined by $w_n:= v_n/||v_n||_{\infty}$ which are $\Sigma_k$-admissible viscosity solutions of 
\begin{equation}\label{wn_solution}
	S_k(D^2 w_n) + \lambda_n w_n |w_n|^{k-1} = \frac{1}{||v_n||_{\infty}} \ \ \text{in} \ \Omega \ \ \ \ \text{and} \ w_n = 0 \ \ \text{on} \ \partial \Omega.
\end{equation}
The uniformly bounded sequence $\{w_n\}_{n \in \N} \subset  C(\overline{\Omega})$ admits a subsequence which converges uniformly on $\overline{\Omega}$ some $\psi_1 \in C(\overline{\Omega})$ which is a $\Sigma_k$-admissible solution of the eigenvalue problem \eqref{EVP} as $\frac{1}{||v_n||_{\infty}} \to 0$ as $n \to \infty$.
\end{proof}

\section{Bounds on the principal eigenvalue.} % $\lambda_1^-(S_k, \Sigma_k)$}

In this section, we will provide an upper bound for the generalized principle eigenvalue $\lambda_1^-(S_k, \Sigma_k)$ as defined in Definition \ref{defn:PEV}. Recall that the lower bound 
\begin{equation}\label{lower_bound_PEV}
	\lambda_1^-(S_k, \Sigma_k) \geq C_{N,k}R^{-2k} \ \ \text{with} \ C_{N,k} = \left( \begin{array}{c} N \\ k \end{array} \right) := \frac{N!}{k! (N - k)!}
\end{equation}
was given in Lemma \ref{lem:PEV} for bounded domains $\Omega$ which are contained in a ball $B_R(0)$. 

An upper bound will be found by constructing a suitable test function which contradicts the minimum principle of Theorem \ref{thm:MPC} on a ball $B_R(0) \subset \Omega$ and makes use of the monotonicity of $\lambda_1^-(S_k, \Sigma_k)$ with respect to set inclusion. More precisely, if we denote by $\lambda_1^-(\Omega)$ the generalized principal eigenvalue $\lambda_1^-(S_k, \Sigma_k)$ with respect to the bounded domain $\Omega$ then one has that
\begin{equation}\label{monotonicity_PEV}
	\Omega^{\prime} \subset \Omega \ \ \Rightarrow \ \ \lambda_1^- (\Omega) \leq \lambda_1^- (\Omega^{\prime}).
\end{equation}
Indeed, since
$$
	\lambda_1^{-}(\Omega):= \sup \{ \lambda \in \R: \ \exists \, \psi \in \Phi_k^{-}(\Omega) \ \text{with} \ S_k(D^2 \psi) + \lambda \psi |\psi|^{k-1} \geq 0 \ \text{in} \ \Omega\},
$$
if $\lambda$ admits such a $\psi$ for $\Omega$ then it also admits $\psi$ for $\Omega^{\prime}$ and hence \eqref{monotonicity_PEV} holds. Our upper bound is contained in the following theorem.

\begin{thm}\label{thm:upper-bound_PEV}
	If a bounded domain $\Omega$ contains the ball $B_R = B_R(0)$, then
\begin{equation}\label{upper_bound_PEV}
		\lambda_1^{-}(\Omega) \leq 4^k \,  C_{N,k} R^{-2k} \ \ \text{with} \ C_{N,k} = \left( \begin{array}{c} N \\ k \end{array} \right) := \frac{N!}{k! (N - k)!}.
\end{equation} 
\end{thm}

\begin{proof} Consider the radial test function (as used in \cite{CNS85})  defined by
\begin{equation}
	u(x) := - \frac{1}{4} \left(R^2 - |x|^2 \right)^2
\end{equation}
and let $r:= |x|$. It suffices to show that $u$ is a $\Sigma_k$-admissible supersolution of
\begin{equation}\label{test_equation}
	S_k(D^2 u) + C u |u|^{k-1} = 0 \ \text{in} \ \Omega
\end{equation}
with $C = 4^k \,  C_{N,k} R^{-2k}$. Indeed, notice that:
\begin{equation}\label{test1}
	u \in C^{\infty}(\R^N) \ \ \text{and hence} \ u \in \LSC(\overline{B}_R);
\end{equation}
\begin{equation}\label{test2}
  u = 0 \ \text{on} \ \partial {B}_R;
\end{equation}
and
\begin{equation}\label{test3}
	\mbox{$B_R \in C^{\infty}$ is strictly $(k-1)$-convex for each $k \in \{1, \ldots , N \}$.}
\end{equation}
However,
\begin{equation}\label{test4}
 	u < 0 \ \text{on} \ \Omega
\end{equation}
 and hence $u$ does not satisfy the minimum principle of of Theorem \ref{thm:MPC} on $B_R$ and hence one must have
$$
	C = 4^k C_{N,k} R^{-2k} \geq \lambda_1^{-}(B_R) \geq \lambda_1^{-}(\Omega),
$$
which would complete the proof.

Since $u \in C^2(B_R)$, it will be a $\Sigma_k$-admissible supersolution of \eqref{test_equation} provided that 
\begin{equation}\label{test_equation2}
	S_k(D^2 u(x)) + C u(x) |u(x)|^{k-1} \leq 0 \ \text{for each} \ x \in B_R,
\end{equation}
as follows easily from Remark \ref{rem:Sk_subsuper} (b). Using the radial formula \eqref{S_k_radial2} with $h(r):= -(R^2-r^2)^2/4$ one finds
$$
	S_k(D^2u) =  (R^2 - r^2)^{k-1} \left( \begin{array}{c} N -1 \\ k -1 \end{array} \right) \left[ \frac{N}{k} (R^2 - r^2) - 2r^2 \right] \leq   (R^2 - r^2)^{k} \left( \begin{array}{c} N  \\ k \end{array} \right) 
$$
and hence
$$
	S_k(D^2u) + Cu|u|^{k-1} \leq (R^2 - r^2)^{k} \left[ \left( \begin{array}{c} N  \\ k \end{array} \right) - \frac{C}{4^k}R^{2k} \right] \leq 0
$$
if
$$
	C \geq 4^k  \left( \begin{array}{c} N  \\ k \end{array} \right) R^{-2k},
$$
as claimed.
\end{proof}

\end{document}